\documentclass[11pt]{amsart}
\usepackage{geometry,graphicx,fancybox}
\usepackage{latexsym}
\usepackage{labelfig}
\usepackage{amssymb}
\usepackage[cp850]{inputenc}
\usepackage{epsfig}
\usepackage{epstopdf}
\usepackage{psfrag}
\usepackage{amsthm}
\usepackage{amscd}
\usepackage{amsmath}
\usepackage{amsfonts}
\usepackage{graphics}
\usepackage{enumerate}

\newtheorem{theorem}{Theorem}[section]
\newtheorem{proposition}[theorem]{Proposition}
\newtheorem{lemma}[theorem]{Lemma}
\newtheorem{corollary}[theorem]{Corollary}

\theoremstyle{definition}
\newtheorem{definition}[theorem]{Definition}
\newtheorem{example}[theorem]{Example}

\newtheorem{remark}[theorem]{Remark}

\newcommand{\sys}{{\rm sys}}
\newcommand{\Sys}{{\mathfrak S}}
\newcommand{\R}{{\mathbb R}}

\newcommand{\N}{{\mathbb N}}

\newcommand{\TT}{{\mathbb T}}

\newcommand{\Z}{{\mathbb Z}}

\newcommand{\BB}{{\mathcal B}}

\newcommand{\G}{{\mathcal G}}
\newcommand{\T}{{\mathcal T}}

\newcommand{\area}{{\rm area}}
\newcommand{\tens}{{\rm tens}}
\newcommand{\arcsinh}{{\,\rm arcsinh}}

\newcommand{\length}{{\rm length}}
\newcommand{\kk}{\Bbbk}
\newcommand{\PP}{\mathcal P}

\newcommand{\cf}{{\it cf.}}
\newcommand{\ie}{{\it i.e.}}

\long\def\forget#1\forgotten{}

\numberwithin{equation}{section}

\begin{document}

\forget
Nothing appears in the document !
\forgotten

%




%

\title{Short loop decompositions of surfaces and the geometry of Jacobians}




\author[F.~Balacheff]{Florent Balacheff}

\address[Florent Balacheff] {Laboratoire Paul Painlev\'e, Universit\'e Lille 1\\ Lille, France}

\email{florent.balacheff@math.univ-lille1.fr}

\author[H.~Parlier]{Hugo Parlier${}^\dagger$}
\address[Hugo Parlier]{Department of Mathematics, University of Fribourg\\
 Switzerland}
\email{hugo.parlier@gmail.com}
\thanks{${}^\dagger$ Research supported by Swiss National Science Foundation grant number PP00P2\textunderscore 128557}

\author[S.~Sabourau]{St\'ephane Sabourau}
\address[St\'ephane Sabourau]{
Laboratoire de Math\'ematiques et Physique Th\'eorique,
Universit\'e Fran\c{c}ois-Rabelais Tours,
F\'ed\'eration Denis Poisson -- CNRS,
Parc de Grandmont, 37200 Tours, France}
\email{sabourau@lmpt.univ-tours.fr}

\date{\today}



\begin{abstract} 
Given a Riemannian surface, we consider a naturally embedded graph which captures part of the topology and geometry of the surface. By studying this graph, we obtain results in three different directions. 

First, we find bounds on the lengths of homologically independent curves on closed Riemannian surfaces.
As a consequence, we show that for any $\lambda \in (0,1)$ there exists a constant $C_\lambda$ such that every closed Riemannian surface of genus $g$ whose area is normalized at $4\pi(g-1)$ has at least $[\lambda g]$ homologically independent loops of length at most $C_\lambda \log(g)$.
This result extends Gromov's asymptotic $\log(g)$ bound on the homological systole of genus~$g$ surfaces.
We construct hyperbolic surfaces showing that our general result is sharp.
We also extend the upper bound obtained by P.~Buser and P.~Sarnak on the minimal norm of nonzero period lattice vectors of Riemann surfaces
in their geometric approach of the Schottky problem to almost $g$ homologically independent vectors.

Then, we consider the lengths of pants decompositions on complete Riemannian surfaces in connexion with Bers' constant and its generalizations.
In particular, we show that a complete noncompact Riemannian surface of genus $g$ with $n$ ends and area normalized to $4\pi (g+\frac{n}{2}-1)$ admits a pants decomposition whose total length (sum of the lengths) does not exceed $C_g \, n \log (n+1)$ for some constant $C_g$ depending only on the genus. 

Finally, we obtain a lower bound on the systolic area of finitely presentable nontrivial groups with no free factor isomorphic to~$\Z$ in terms of its first Betti number. The asymptotic behavior of this lower bound is optimal.
\end{abstract}




%



\maketitle


\section{Introduction}

Consider a surface of genus $g$ with $n$ marked points and consider the different complete Riemannian metrics of finite area one can put on it. These include complete hyperbolic metrics of genus $g$ with $n$ cusps, but also complete Riemannian metrics with $n$ ends which we normalize to the area of their hyperbolic counterparts.
We are interested in describing what surfaces with large genus and/or large number of ends can look like. More specifically, we are interested in the lengths of certain curves that help describe the geometry of the surface.

\

In the first part of this article, we generalize the following results regarding the shortest length of a homologically nontrivial loop on a closed Riemannian surface (i.e., the homological systole) and on the Jacobian of a Riemann surface.
The homological systole of a closed Riemannian surface of genus $g$ with normalized area, that is, with area $4 \pi (g-1)$, is at most $\sim \log(g)$.
This result, due to M.~Gromov \cite[2.C]{gro96}, is optimal.
Indeed, there exist families of hyperbolic surfaces, one in each genus, whose homological systoles grow like $\sim \log(g)$.
The first of these were constructed by P.~Buser and P.~Sarnak in their seminal article~\cite{BS94}, and there have been other constructions since by R. Brooks \cite{BR99} and M.~Katz, M.~Schaps and U.~Vishne \cite{KSV07}.
By showing that the shortest homologically nontrivial loop on a hyperbolic surface lies in a ``thick" embedded cylinder, P.~Buser and P.~Sarnak also derived new bounds on the minimal norm of nonzero period lattice vectors of Riemann surfaces. 
This result paved the way for a geometric approach of the Schottky problem which consists in characterizing Jacobians (or period lattices of Riemann surfaces) among abelian varieties.

\

Bounds on the lengths of curves in a homology basis have also been studied by P.~Buser and M.~Sepp\"al\"a~\cite{BuS02,BS03} for closed hyperbolic surfaces.
Note however that without a lower bound on the homological systole, the $g+1$ shortest homologically independent loop cannot be bounded by any function of the genus. Indeed, consider a hyperbolic surface with $g$ very short homologically independent (and thus disjoint) loops. Every loop homologically independent from these short curves must cross one of them, and via the collar lemma, can be made arbitrarily large by pinching our initial $g$ curves. 

\

As a preliminary result, we obtain new bounds on the lengths of short homology basis for closed Riemannian surfaces with homological systole bounded from below.

\begin{theorem} \label{theo:0}
Let $M$ be a closed orientable Riemannian surface of genus~$g$ with homological systole at least~$\ell$ and area equal to $4\pi(g-1)$.
Then there exist $2g$ loops $\alpha_{1},\ldots,\alpha_{2g}$ on~$M$ which induce a basis of~$H_{1}(M;\Z)$ such that
\begin{equation}
\length(\alpha_{k}) \leq C_{0} \, \frac{\log(2g-k+2)}{2g-k+1} \, g,
\end{equation}
where $C_0=\frac{2^{16}}{\min\{1,\ell\}}$.

\end{theorem}

On the other hand, without assuming any lower bound on the homological systole, M.~Gromov~\cite[1.2.D']{gro83} proved that on every closed Riemannian surface of genus~$g$ with area normalized to $4 \pi (g-1)$, the length of the $g$ shortest homologically independent loops is at most $\sim \sqrt{g}$. Furthermore, Buser's so-called hairy torus example~\cite{bus81,bus92} with hair tips pairwise glued together shows that this bound is optimal, even for hyperbolic surfaces.

\

A natural question is then to find out for how many homologically independent curves does Gromov's $\log(g)$ bound hold. The only result in this direction we are aware of is due to B. Muetzel~\cite{mue}, who recently proved that on every genus~$g$ hyperbolic surface there exist at least two homologically independent loops of lengths at most $\sim \log(g)$. We show that on every closed Riemannian surface of genus~$g$ with normalized area there exist almost $g$ homologically independent loops of lengths at most $\sim \log(g)$. More precisely, we prove the following.

\begin{theorem} \label{theo:A}
Let $\eta:\N \to \N$ be a function such that 
$$
\displaystyle \lambda := \sup_g \frac{\eta(g)}{g} < 1.
$$
Then there exists a constant~$C_\lambda$ such that for every closed Riemannian surface~$M$ of genus~$g$ there are at least $\eta(g)$ homologically independent loops $\alpha_{1},\ldots,\alpha_{\eta(g)}$ which satisfy
\begin{equation*}
\length(\alpha_{i}) \leq C_\lambda \, \frac{\log(g+1)}{\sqrt{g}} \, \sqrt{\area(M)}
\end{equation*}
for every $i \in \{1,\ldots,\eta(g) \}$.
\end{theorem}
Typically, this result applies to $\eta(g)=[\lambda g]$ where $\lambda \in (0,1)$.

\

Thus, the previous theorem generalizes Gromov's $\log(g)$ bound on the homological systole, \cf~\cite[2.C]{gro96}, to the lengths of almost $g$ homologically independent loops.
Note that its proof differs from other systolic inequality proofs.
Specifically, it directly yields a $\log(g)$ bound on the homological systole without considering the homotopical systole (that is, the shortest length of a homotopically nontrivial loop).
Initially, M.~Gromov obtained his bound from a similar bound on the homotopical systole using surgery, \cf~\cite[2.C]{gro96}.
However the original proof of the $\log(g)$ bound on the homotopical systole, \cf~\cite[6.4.D']{gro83} and \cite{gro96}, as well as the alternative proofs available, \cf~\cite{bal04,KS05}, do not directly apply to the homological systole.

\medskip

One can ask how far from being optimal our result on the number of short (homologically independent) loops is.
Of course, in light of the Buser-Sarnak examples, one can not hope to do (roughly) better than a logarithmic bound on their lengths, but the question on the number of such curves remains.
Now, because of Buser's hairy torus example, we know that the $g$ shortest homologically independent loops of a hyperbolic surface of genus~$g$ can grow like $\sim \sqrt{g}$ and that the result of Theorem~\ref{theo:A} cannot be extended to $\eta(g)=g$.
Still, one can ask for $g-1$ homologically independent loops of lengths at most $\sim \log(g)$, or for a number of homologically independent loops of lengths at most $\sim \log(g)$ which grows asymptotically like~$g$.
Note that the surface constructed from Buser's hairy torus does not provide a counterexample in any of these cases.

\medskip

Our next theorem shows this is impossible, which proves that the result of Theorem~\ref{theo:A} on the number of homologically independent loops whose lengths satisfy a $\log(g)$ bound is optimal. Before stating this theorem, it is convenient to introduce the following definition.

\begin{definition} \label{def:ksys}
Given $k \in \N^*$, the \emph{$k$-th homological systole} of a closed Riemannian manifold~$M$, denoted by~$\sys_{k}(M)$, is defined as the smallest real~$L \geq 0$ such that there exist $k$ homologically independent loops on~$M$ of length at most~$L$.
\end{definition}

With this definition, under the assumption of Theorem~\ref{theo:A}, every closed Riemannian surface of genus~$g$ with area $4 \pi (g-1)$ satisfies
$$
\sys_{\eta(g)}(M) \leq C_\lambda \, \log(g+1)
$$
for some constant~$C_\lambda$ depending only on~$\lambda$.
Furthermore, still under the assumption of Theorem~\ref{theo:A}, Gromov's sharp estimate, \cf~\cite[1.2.D']{gro83}, with this notation becomes
$$
\sys_g(M) \leq C \, \sqrt{g}
$$
where $C$ is a universal constant.

\

We can now state our second main result.

\begin{theorem} \label{theo:B}
Let $\eta:\N \to \N$ be a function such that 
$$
\displaystyle \lim_{g \to \infty} \frac{\eta(g)}{g} = 1.
$$
Then there exists a sequence of genus~$g_{k}$ hyperbolic surfaces~$M_{g_{k}}$ with $g_{k}$ tending to infinity such that
$$
\lim_{k \to \infty} \frac{\sys_{\eta(g_{k})}(M_{g_{k}})}{\log(g_{k})} = \infty.
$$
\end{theorem}


In their geometric approach of the Schottky problem, P.~Buser and P.~Sarnak~\cite{BS94} also proved that the homological systole of the Jacobian of a Riemann surface~$M$ of genus~$g$ is at most $\sim \sqrt{\log(g)}$ and this bound is optimal.
In other words, there is a nonzero lattice vector in~$H^1(M;\Z)$ whose $L^2$-norm satisfies a $\sqrt{\log(g)}$ upper bound (see Section~\ref{sec:jacobian} for a precise definition). We extend their result by showing that there exist almost $g$ linearly independent lattice vectors whose norms satisfy a similar upper bound.
More precisely, we have the following.

\begin{corollary} \label{coro:C}
Let $\eta:\N \to \N$ be a function such that 
$$
\displaystyle \lambda := \sup_g \frac{\eta(g)}{g} < 1.
$$
Then there exists a constant~$C_\lambda$ such that for every closed Riemann surface~$M$ of genus~$g$ there are at least $\eta(g)$ linearly independent lattice vectors $\Omega_{1},\ldots,\Omega_{\eta(g)} \in H^1(M;\Z)$ which satisfy
\begin{equation}
|\Omega_i|_{L^2}^2 \leq C_\lambda \, \log(g+1)
\end{equation}
for~every $i \in \{1,\ldots,\eta(g)\}$.
\end{corollary}

Contrary to Theorem~\ref{theo:A}, we do not know whether the result of Corollary~\ref{coro:C} is sharp regarding the number of independent lattice vectors of norm at most $\sim \sqrt{\log(g)}$.

\medskip

To prove Theorem~\ref{theo:0} we consider naturally embedded graphs which capture a part of the topology and geometry of the surface, and study these graphs carefully.
Then, we derive Theorem~\ref{theo:A} in the absence of lower bound on the homological systole.
We actually present the proof of Theorem~\ref{theo:A}, first in the hyperbolic case (restricting ourselves to hyperbolic metrics in our constructions), then in the Riemannian case (a more general framework which allows us to make use of more flexible constructions).
To derive Corollary~\ref{coro:C}, we show that on closed hyperbolic surfaces the loops given by Theorem~\ref{theo:A} have embedded collars of uniform width.
To prove Theorem~\ref{theo:B}, we adapt known constructions of surfaces with large homological systole to obtain closed hyperbolic surfaces of large genus which asymptotically approach the limit case.\\

In the second part of this article, we study an invariant related to pants decompositions of surfaces. A pants decomposition is a collection of nontrivial disjoint simple loops on a surface so that the complementary region is a set of three-holed spheres (so-called pants).
L.~Bers~\cite{bers1,bers2} showed that on a hyperbolic surface one can always find such a collection with all curves of length bounded by a constant which only depends on the genus and number of cusps. These constants are generally called Bers' constants. This result was quantified and generalized to closed Riemannian surfaces of genus~$g$ by P.~Buser \cite{bus81,bus92}, and P.~Buser and M. S\"eppala \cite{BS92} who showed that the constants behave at least like $\sim \sqrt{g}$ and at most like $\sim g$. The correct behavior remains unknown, except in the cases of punctured spheres and hyperelliptic surfaces \cite{BP09}.

\medskip

We apply the same graph embedding technique as developed in the first part to obtain results on the minimal total length of such pants decompositions of surfaces. For a complete Riemannian surface of genus $g$ with $n$ ends whose area is normalized at $4\pi(g+\frac{n}{2}-1)$, P.~Buser's bounds mentioned before imply that one can always find a pants decomposition whose sum of lengths is bounded from above by $\sim (g+n)^2$. Our main result is the following.

\begin{theorem}
Let $M$ be a complete noncompact Riemannian surface of genus $g$ with $n$ ends whose area is equal to $4\pi(g+\frac{n}{2}-1)$.
Then $M$ admits a pants decomposition whose sum of the lengths is bounded from above by
$$
C_g\, n \log (n+1),
$$
where $C_g$ is an explicit genus dependent constant.
\end{theorem}

In the case of a punctured sphere, it suffices to study the embedded graph mentioned above. In the more general case, this requires a bound on the usual Bers' constant which relies on F.~Balacheff and S.~Sabourau's bounds on the diastole, \cf~\cite{BS10}. Specifically, in Proposition \ref{prop:bers} we show that the diastole is an upper bound on lengths of pants decompositions, and thus, if one is not concerned with the multiplicative constants, the result of~\cite{BS10} provides an optimal square root upper-bound on Bers' constants for puncture growth, and an alternative proof of Buser's linear bounds for genus growth.

\medskip

As a corollary to the above we show that a hyperelliptic surface of genus $g$ admits a pants decomposition of total length at most $\sim g\log g$. This is in strong contrast with the general case according to a result of L.~Guth, H.~Parlier and R.~Young~\cite{GPY}: ``random" hyperbolic surfaces have all their pants decompositions of total length at least $\sim g^{\frac{7}{6}-\varepsilon}$ for any $\varepsilon>0$. \\

In the last part of this article, we consider the systolic area of finitely presentable groups, \cf~Definition~\ref{def:SG}.
From~\cite[6.7.A]{gro83}, the systolic area of an unfree group is bounded away from zero, see also \cite{KRS06}, \cite{RS08}, \cite{KKSSW} and \cite{BB} for simpler proofs and extensions.
The converse is also true: a finitely presentable group with positive systolic area is not free, \cf~\cite{KRS06}.
The systolic finiteness result of~\cite{RS08} regarding finitely presentable groups and their structure provides a lower bound on the systolic area of these groups in terms of their first Betti number when the groups have no free factor isomorphic to~$\Z$.
In particular, the systolic area of such a group $G$ goes to infinity with its first Betti number~$b_1(G)$.
Specifically,
$$
\Sys(G) \geq C \, (\log(b_1(G)+1))^{\frac{1}{3}}
$$
where $\Sys(G)$ is the systolic area of~$G$ and $C$ is some positive universal constant.

\medskip

In this article, we improve this lower bound.

\begin{theorem} \label{theo:G}
Let $G$ be a finitely presentable nontrivial group with no free factor isomorphic to~$\Z$.
Then
\begin{equation} \label{eq:GG}
\Sys(G)\geq C \, \frac{b_1(G)+1}{(\log(b_1(G)+2))^2}
\end{equation}
for some positive universal constant $C$.
\end{theorem}

In Example \ref{ex:GG}, we show that the order of the bound in inequality~\eqref{eq:GG} cannot be improved. The proof of Theorem~\ref{theo:G} follows the proof of Theorem~\ref{theo:0} in this somewhat different context. 

\section{Independent loops on surfaces with homological systole bounded from below} \label{sec:ell}

Here we show the following theorem which allows us to bound the lengths of an integer homology basis in terms of the genus and the homological systole of the surface.

\begin{theorem} \label{theo:ell}
Let $M$ be a closed orientable Riemannian surface of genus~$g$ with homological systole at least~$\ell$ and area equal to $4\pi(g-1)$.
Then there exist $2g$ loops $\alpha_{1},\ldots,\alpha_{2g}$ on~$M$ which induce a basis of~$H_{1}(M;\Z)$ such that
\begin{equation} \label{eq:gen}
\length(\alpha_{k}) \leq C_{0} \, \frac{\log(2g-k+2)}{2g-k+1} \, g,
\end{equation}
where $C_0=\frac{2^{16}}{\min\{1,\ell\}}$.\\

\noindent In particular:
\begin{enumerate}
\item the lengths of the $\alpha_{i}$ are bounded by~$C_{0} \, g$; \label{item1}
\item the median length of the $\alpha_{i}$ is bounded by~$C_{0} \, \log(g+1)$. \label{item2}
\end{enumerate}
\end{theorem}

\begin{remark}
The linear upper bound in the genus of item~\eqref{item1} already appeared in~\cite{BS03} for hyperbolic surfaces, where the authors obtained a similar bound for the length of so-called canonical homology basis.
They also constructed a genus~$g$ hyperbolic surface all of whose homology bases have a loop of length at least~$C \, g$ for some positive constant~$C$.
This shows that the linear upper bound in~\eqref{item1} is roughly optimal.
However, the general bound~\eqref{eq:gen} on the length of the loops of a short homology basis, and in particular the item~\eqref{item2}, cannot be derived from the arguments of~\cite{BS03} even in the hyperbolic case.

The bound obtained in~\eqref{item2} is also roughly optimal.
Indeed, the Buser-Sarnak surfaces~\cite{BS94} have their homological systole greater or equal to $\frac{4}{3} \log(g)$ minus a constant.
\end{remark}

\begin{remark}
In a different direction, the homological systole of a ``typical" hyperbolic surface, where we take R.~Brooks and E.~Makover's definition of a random surface \cite{BM04}, is bounded away from zero. Therefore, the conclusion of the theorem holds for these typical hyperbolic surfaces with $\ell$ constant. However, their diameter is bounded by $C \log(g)$ for some constant~$C$. This shows that there exists a homology basis on these surfaces formed of loops of length at most $2C \log(g)$ (see Remark \ref{rem:diameter}). Thus, the upper bound in~\eqref{item1} is not optimal for these surfaces.
\end{remark}

\begin{remark}\label{rem:ell}
A non-orientable version of this theorem also holds. Recall that a closed non-orientable surface of genus $g$ is a surface homeomorphic to the connected sum of $g$ copies of the projective plane. Let $M$ be a closed non-orientable Riemannian surface of genus~$g$ with homological systole at least~$\ell$. Then there exist $g$ loops $\alpha_{1},\ldots,\alpha_{g}$ on~$M$ which induce a basis of~$H_{1}(M;\Z)$ such that
$$
\length(\alpha_{k}) \leq C_{0} \, \frac{\log(g-k+2)}{g-k+1} \, g,
$$
where $C_0=\frac{C}{\min\{1,\ell\}}$ for some positive constant $C$.
\end{remark}

Let us introduce some definitions and results, which will be used several times in this article.

\begin{definition} \label{def:min}
Let $(\gamma_i)_i$ be a collection of loops on a compact Riemannian surface~$M$ of genus~$g$ (possibly with boundary components).
The loops~$(\gamma_i)_i$ form a \emph{minimal homology basis} of~$M$ if
\begin{enumerate}
\item their homology classes form a basis of $H_1(M;\Z)$;
\item for every $k=1,\ldots,2g$ and every collection of $k$ homologically independent loops~$\gamma'_1,\ldots,\gamma'_k$, there exist $k$ loops $\gamma_{i_1},\ldots,\gamma_{i_k}$ of length at most 
$$
\sup_{1 \leq j \leq k} \length(\gamma'_j).
$$
\label{2}
\end{enumerate}
In this definition, one could replace the condition~\eqref{2} with the following condition:
\begin{enumerate}
\item[(2')] for every collection of loops~$(\gamma'_i)_i$ whose homology classes form a basis of $H_1(M;\Z)$, we have
$$
\sum_{i=1}^{2g} \length(\gamma_i) \leq \sum_{i=1}^{2g} \length(\gamma'_i).
$$
\end{enumerate}
\end{definition}

\begin{lemma} \label{lem:cut}
Let $M$ be a closed orientable Riemannian surface.
Let $\alpha$ be a homologically trivial loop such that the distance between any pair of points of~$\alpha$ is equal to the length of the shortest arc of~$\alpha$ between these two points.
Then no curve of a minimal homology basis of~$M$ crosses~$\alpha$.
\end{lemma}

\begin{proof}
Let~$\gamma$ be a loop of~$M$ which crosses~$\alpha$.
Then $\gamma$ must cross~$\alpha$ at least twice. Consider an arc $c$ of $\gamma$ that leaves from $\alpha$ and then returns. Let $d$ be the shortest arc of~$\alpha$ connecting the two endpoints of~$c$. By assumption, $d$ is no longer than $c$ and $\gamma \setminus c$. Thus, both $c \cup d$ and $(\gamma \setminus c) \cup d$ are homotopic to loops which are shorter than~$\gamma$. These loops are obtained by smoothing out $c \cup d$ and $(\gamma \setminus c) \cup d$. Since $\gamma$ is homologuous to the sum of these loops, with proper orientations, the curve~$\gamma$ does not lie in a minimal homology basis.
\end{proof}

We continue with some more notations and definitions. We denote $\ell_{M}(\gamma)$ the infimal length of the loops of~$M$ freely homotopic to~$\gamma$.

\begin{definition} \label{def:mls}
Consider two metrics on the same surface and denote by $M$ and $M'$ the two metric spaces.
The marked length spectrum of $M$ is said to be greater or equal to the marked length spectrum of~$M'$ if
\begin{equation} \label{eq:mls}
\ell_M(\gamma) \geq \ell_{M'}(\gamma)
\end{equation}
for every loop~$\gamma$ on the surface.
Similarly, the two marked length spectra of $M$ and $M'$ are equal if equality holds in~\eqref{eq:mls} for every loop~$\gamma$ on the surface.
\end{definition}

\begin{definition} \label{def:sys}
Let $M$ be a compact nonsimply connected Riemannian manifold.

The \emph{homotopical systole} of~$M$, denoted by $\sys_{\pi}(M)$, is the length of its shortest noncontractible loop of~$M$.
A homotopical systolic loop of~$M$ is a noncontractible loop of~$M$ of least length.

Similarly, the \emph{homological systole} of~$M$, denoted by $\sys_{H}(M)$, is the length of its shortest homologically nontrivial loop of~$M$.
A homological systolic loop of~$M$ is a homologically nontrivial loop of~$M$ of least length.
\end{definition}

Note that $\sys_H(M) = \sys_1(M)$, \cf~Definition~\ref{def:ksys}.

\

In order to prove Theorem \ref{theo:ell}, we will need the following result from~\cite[5.6.C'']{gro83}.

\begin{lemma} \label{lem:regmet}
Let $M_{0}$ be a closed Riemannian surface and $0<R \leq \frac{1}{2} \, \sys_{\pi}(M_{0})$.
Then there exists a closed Riemannian surface $M$ conformal to $M_0$ such that
\begin{enumerate}
\item $M$ and $M_{0}$ have the same area;
\item the marked length spectrum of $M$ is greater or equal to that of $M_0$;
\item the area of every disk of radius $R$ in $M$ is greater or equal to~$R^{2}/2$.
\end{enumerate}
\end{lemma}

We can now proceed to the proof of Theorem~\ref{theo:ell}.

\begin{proof}[Proof of Theorem~\ref{theo:ell}] \mbox{ }

\medskip

Without loss of generality, we can suppose that $\ell$ is at most~$1$.

\medskip

\noindent {\it Step 1.} 
Let us show first that we can assume that the homotopical systole of~$M$, and not merely the homological systole, is bounded from below by~$\ell$.
Suppose that there exists a homotopical systolic loop~$\alpha$ of~$M$ of length less than~$\ell$ which is homologically trivial.
Clearly, the distance between any pair of points of~$\alpha$ is equal to the length of the shortest arc of~$\alpha$ connecting this pair of points. (Note that both the homotopical and homological systolic loops have this property for closed surfaces.) We split the surface along the simple loop~$\alpha$ and attach two round hemispheres along the boundary components of the connected components of the surface.
We obtain two closed surfaces $M'$ and~$M''$ of genus less than~$g$.
The sum of their areas is equal to $\area(M)+\frac{\ell^2}{\pi}$.

Collapsing $\alpha$ to a point induces an isomorphism $H_1(M;\Z) \to H_1(M' \vee M'';\Z) \simeq H_1(M' \coprod M'';\Z)$.
From Lemma~\ref{lem:cut}, the loops of a minimal homology basis of~$M$ do not cross~$\alpha$.
Therefore, they also lie in the disjoint union $M' \coprod M''$.
Conversely, every loop of $M' \coprod M''$ can be deformed without increasing its length into a loop which does not go through the two round hemispheres (that we previously attached), and therefore also lies in~$M$.
This shows that two minimal homology basis of~$M$ and $M' \coprod M''$ have the same lengths.

We repeat the previous surface splitting to the new surfaces on and on as many times as possible.
After at most~$g$ steps (\ie, $g$ cuts), this process stops.
By construction, we obtain a closed Riemannian surface~$N$ with several connected components such that
\begin{enumerate}
\item the homotopical systole of~$N$ is at least~$\ell$;
\item $H_1(N;\Z)$ is naturally isomorphic to $H_1(M;\Z)$;
\item two minimal homology basis of~$M$ and $N$ have the same length.
\end{enumerate}
Furthermore, we have
\begin{eqnarray*}
\area(N) & \leq &\area(M) + g \, \frac{\ell^2}{\pi} \\
& \leq & 2^4 (1+ \ell^2) \, g.
\end{eqnarray*}
Thus, it is enough to show that the conclusion of Theorem~\ref{theo:ell} holds for every (not necessarily connected) closed Riemannian surface~$M$ of genus~$g$ with homotopical systole at least~$\ell$ and area at most $2^4 (1+ \ell^2) \, g$. \\

\noindent{\it Step 2.}
Assume now that $M$ is such a surface.
Fix $r_0 < \ell/8$, say $r_0=\ell/16$.
By Lemma \ref{lem:regmet}, we can suppose that any disk of radius~$r_{0}$ of~$M$ has area at least~$r_{0}^2/2$.
Consider a maximal system of disjoint disks $\{D_i\}_{i \in I}$ of radius~$r_{0}$.
Since each disk~$D_i$ has area at least $r_{0}^2/2$, the system admits at most $2\, \area(M)/r_{0}^2$ disks.
That is,
\begin{eqnarray} 
|I| & \leq & \frac{2^5}{r_{0}^2} \, (1+ \ell^2) \, g \nonumber \\
& \leq & 2^{13} \, (1 + \frac{1}{\ell^2}) \, g. \label{eq:I}
\end{eqnarray}
As this system is maximal, the disks $2D_{i}$ of radius~$2r_{0}$ with the same centers~$x_{i}$ as~$D_{i}$ cover~$M$. \\

\noindent Let $2D_{i}+\varepsilon$ be the disks centered at~$x_{i}$ with radius~$2r_{0}+\varepsilon$, where $\varepsilon>0$ satisfies $4r_{0}+2\varepsilon<\ell/2 \leq \sys_{\pi}(M)/2$.
Consider the $1$-skeleton~$\Gamma$ of the nerve of the covering of~$M$ by the disks~$2D_{i}+\varepsilon$.
In other words, $\Gamma$ is a graph with vertices~$\{v_{i}\}_{i \in I}$ corresponding to the centers ~$\{x_{i}\}_{i \in I}$ where two vertices $v_{i}$ and~$v_{j}$ are connected by an edge if and only if $2D_{i}+\varepsilon$ and~$2D_{j}+\varepsilon$ intersect each other.
Denote by $v$, $e$ and $b$ its number of vertices, its number of edges and its first Betti number.
We have the relation $b \geq e-v+1$ (with equality if the graph is connected). \\

\noindent Endow the graph $\Gamma$ with the metric such that each edge has length $\ell/2$.
Consider the map \mbox{$\varphi:\Gamma \to M$} which takes each edge with endpoints $v_{i}$ and~$v_{j}$ to a geodesic segment connecting $x_{i}$ and~$x_{j}$.
This segment is not necessarily unique, but we can choose one.
Since the points~$x_{i}$ and~$x_{j}$ are distant from at most $4r_{0}+2\varepsilon<\ell/2$, the map $\varphi$ is distance nonincreasing.

\begin{lemma} \label{lem:epi}
The map $\varphi$ induces an epimorphism $\pi_{1}(\Gamma) \to \pi_{1}(M)$.

In particular, it induces an epimorphism in integral homology.
\end{lemma}

\begin{proof}[Proof of Lemma \ref{lem:epi}]
Let $c$ be a geodesic loop of~$M$.
Divide $c$ into segments $c_{1}, \ldots, c_{m}$ of length less than~$\varepsilon$.
Denote by $p_{k}$ and~$p_{k+1}$ the endpoints of~$c_{k}$ with $p_{m+1}=p_{1}$.
Since the disks~$2D_{i}$ cover~$M$, every point~$p_{k}$ is at distance at most~$2r_{0}$ from a point~$q_{k}$ among the centers~$x_{i}$.
Consider the loop
$$
\alpha_k= c_k \cup [p_{k+1},q_{k+1}] \cup [q_{k+1},q_k] \cup [q_k,p_k],
$$
where $[x,y]$ denotes a segment joining $x$ to~$y$. 
Then
$$
\length(\alpha_k) \leq 2 \, (4 r_{0} + \varepsilon) < \sys_{\pi}(M).
$$ 
Thus, the loops $\alpha_k$ are contractible.
Therefore, the loop~$c$ is homotopic to a piecewise geodesic loop~$c'= (q_1, \dots, q_{m})$.

Since the distance between the centers $q_{k}=x_{i_{k}}$ and $q_{k+1}=x_{i_{k+1}}$ is at most $4r_{0}+\varepsilon$, the vertices $v_{i_{k}}$ and~$v_{i_{k+1}}$ are connected by an edge in~$\Gamma$.
The union of these edges forms a loop~$(v_{i_{1}}, \ldots, v_{i_{m}})$ in~$\Gamma$ whose image by~$\varphi$ agrees with~$c'$ and is homotopic to the initial loop~$c$.
Hence the result.
\end{proof}

\noindent Let $\kk$ be a field.
Consider a subgraph~$\Gamma_{1}$ of~$\Gamma$ with a minimal number of edges such that the restriction of~$\varphi$ to~$\Gamma_{1}$ still induces an epimorphism in homology with coefficients in~$\kk$.
The graph~$\Gamma_{1}$ inherits the simplicial structure of~$\Gamma$.

\begin{lemma} \label{lem:isom}
The epimorphism $\varphi_{*}:H_{1}(\Gamma_{1};\kk) \to H_{1}(M;\kk)$ induced by~$\varphi$ is an isomorphism.

In particular, the first Betti number~$b_{1}$ of~$\Gamma_{1}$ is equal to~$2g$.
\end{lemma}

\forget
\begin{proof}[Proof of Lemma \ref{lem:isom}]
Suppose the contrary and remove an edge from a simple cycle~$c$ representing a nontrivial element of the kernel of~$\varphi_{*}$.
The resulting graph~$\Gamma'_{1}$ has fewer edges than~$\Gamma_{1}$.
Adding a suitable integral multiple of~$c$ to any cycle~$\sigma$ of~$\Gamma_{1}$ yields a new cycle~$\sigma'$ lying in~$\Gamma'_{1}$.
The cycles $\sigma$ and~$\sigma'$ are sent to the same homology class by~$\varphi$.
Thus, the restriction of~$\varphi$ to~$\Gamma'_{1}$ still induces an epimorphism in integral homology, which is absurd by definition of~$\Gamma_{1}$.
\end{proof}
\forgotten

\begin{proof}[Proof of Lemma \ref{lem:isom}]
The graph~$\Gamma_1$ is homotopy equivalent to a union of bouquets of circles $c_1,\cdots,c_m$ (simply identify a maximal tree in each connected component of~$\Gamma_1$ to a point).
The homology classes $[c_1],\cdots,[c_m]$ of these circles form a basis in homology with coefficients in~$\kk$.
If the image by~$\varphi_*$ of one of these homology classes~$[c_{i_0}]$ lies in the vector space spanned by the images of the others, we remove the edge of~$\Gamma_1$ corresponding to the circle~$c_{i_0}$.
The resulting graph~$\Gamma'_{1}$ has fewer edges than~$\Gamma_1$ and the restriction of~$\varphi$ to~$\Gamma'_{1}$ still induces an epimorphism in homology with coefficients in~$\kk$, which is absurd by definition of~$\Gamma_{1}$.
\end{proof}

\noindent At least $b-b_{1}$ edges were removed from $\Gamma$ to obtain~$\Gamma_{1}$.
As the length of every edge of~$\Gamma$ is equal to~$\ell/2$, we have
\begin{eqnarray}
\length(\Gamma_{1}) & \leq & \length(\Gamma) - (b-b_{1}) \, \frac{\ell}{2} \nonumber \\
 & \leq & (e-b+2g) \, \frac{\ell}{2} \nonumber \\
 & \leq & (v-1+2g) \, \frac{\ell}{2} \label{eq:l1} 
\end{eqnarray}

\noindent Let us contruct by induction $n$ graphs~$\Gamma_{n} \subset \cdots \subset \Gamma_{1}$ and $n$ (simple) loops~$\displaystyle \{ \gamma_{k} \}_{k=1}^{n}$ with $n \leq 2g$ such that
\begin{enumerate}
\item $\gamma_{k}$ is a systolic loop of~$\Gamma_{k}$;
\item the loops~$(\gamma_{k})_{k=1}^{n-1}$ induce a basis of a supplementary subspace of~$H_{1}(\Gamma_{n};\kk)$ in~$H_{1}(\Gamma_{1};\kk)$.
\end{enumerate}
For $n=1$, the result clearly holds. \\

\noindent Suppose we have constructed $n$ graphs~$\{\Gamma_{k}\}_{k=1}^{n}$ and $n$ loops~$\{\gamma_{k}\}_{k=1}^{n}$ satisfying these properties.
Remove an edge of~$\Gamma_{n}$ through which $\gamma_{n}$ passes.
We obtain a graph~$\Gamma_{n+1} \subset \Gamma_{n}$ such that $H_{1}(\Gamma_{n};\kk)$ decomposes into the direct sum of $H_{1}(\Gamma_{n+1};\kk)$ and~$\kk \, [\gamma_{n}]$, where $[\gamma_{n}]$ is the homology class of~$\gamma_{n}$ (recall that $\gamma_{n}$ generates a nontrivial class in homology).
That is,
\begin{equation} \label{eq:decomp}
H_{1}(\Gamma_{n};\kk) = H_{1}(\Gamma_{n+1};\kk) \oplus \kk \, [\gamma_{n}].
\end{equation}
Let $\gamma_{n+1}$ be a systolic loop of~$\Gamma_{n+1}$.
The condition on~$\{ \gamma_{k} \}_{k=1}^{n}$ along with the decomposition~\eqref{eq:decomp} shows that the loops~$(\gamma_{k})_{k=1}^{n}$ induce a basis of a supplementary subspace of~$H_{1}(\Gamma_{n+1};\kk)$ in~$H_{1}(\Gamma_{1};\kk)$.
This concludes the construction by induction. \\

\noindent The graph~$\Gamma_{k}$ has $k-1$ fewer edges than~$\Gamma_{1}$.
Hence, from~\eqref{eq:l1}, we deduce that
\begin{eqnarray}
\length(\Gamma_{k}) & \leq & \length(\Gamma_{1}) - (k-1) \, \frac{\ell}{2} \nonumber \\
 & \leq & (v+2g - k) \, \frac{\ell}{2}. \label{eq:vg}
\end{eqnarray}

\noindent By construction, the first Betti number~$b_{k}$ of~$\Gamma_{k}$ is equal to $b_{1}-k+1=2g-k+1$.
Thus, Bollob\'as-Szemer\'edi-Thomason's systolic inequality on graphs~\cite{BT97,BS02} along with the bounds \eqref{eq:vg} and~\eqref{eq:I} implies that
\begin{eqnarray*}
\length(\gamma_{k}) = \sys(\Gamma_{k}) & \leq & 4 \, \frac{\log(1+b_{k})}{b_{k}} \, \length(\Gamma_{k}) \nonumber \\
& \leq & 2^{15} \, \left( \ell + \frac{1}{\ell} \right) \, \frac{\log(2g-k+2)}{2g-k+1} \, g. 
\end{eqnarray*}
Hence, 
$$
\length(\gamma_{k}) \leq C_{0} \, \frac{\log(2g-k+2)}{2g-k+1} \, g,
$$
where $\displaystyle C_{0} = \frac{2^{16}}{\min\{1,\ell\}}$ (recall we can assume that $\ell \leq 1$). 

Since the map $\varphi$ is distance nonincreasing and induces an isomorphism in homology, the images of the loops~$\gamma_{k}$ by~$\varphi$ yield the desired curves~$\alpha_k$ on~$M$. \\

Now, recall that
$$
\length(\gamma_{1}) \leq \cdots \leq \length(\gamma_{2g}).
$$
We deduce that the curves~$\alpha_k$ are of length at most
$$
\length(\gamma_{2g}) \leq C_0 \, g
$$
and that the median length is bounded from above by
$$
\length(\gamma_{g+1}) \leq C_0 \, \log(g+1). 
$$
This concludes the proof of Theorem \ref{theo:ell}.
\end{proof}

\forget
\begin{remark} \label{rem:C0}
When the metric is hyperbolic, the disks~$D_i$ are embedded and have area \mbox{$2 \pi \, (\cosh(r_0)-1)$}.
In this case, we can take $C_0= \frac{8}{s_0} \, (\ell+\frac{1}{\ell})$, where $s_0 = \cosh(r_0)-1$.
\end{remark}

\bigskip

We conclude this section with some consequences concerning hyperbolic surfaces. \\

The existence of hyperbolic surfaces whose homological systole satisfies a near-optimal asymptotic behaviour in~$\log(g)$ follows from the arithmetic construction of Buser-Sarnak~\cite{BS94}, extended in Katz-Schnider-Vishne~\cite{KSV07}.
More specifically, these constructions yield hyperbolic surfaces of genus~$g$ (in certain genus) whose homological systole is bounded from below by $\frac{4}{3} \log(g)-c$, where $c$ is a constant. \\

Keeping in mind this result, the previous theorem leads to the following bound on the length of short homology basis on these hyperbolic surfaces with large homological systole. 

\begin{corollary} \label{prop:BS}
Given a real $c$, every closed hyperbolic orientable surface of genus~$g$ whose homological systole is bounded from below by $\frac{4}{3} \log(g)-c$ admits a homology basis consisting of loops of length at most
$$
C_{1} \, g^{\frac{5}{6}} \, \log g,
$$
where $C_1$ is a constant depending only on~$c$. 
\end{corollary}

\begin{proof}
The estimate follows from the bound on the length of~$\gamma_{2g}$ in Theorem~\ref{theo:ell}.\eqref{item1}.
Simply use the formula for~$C_0$ given in Remark~\ref{rem:C0} and plug in $\ell = \frac{4}{3} \log(g)-c$ with $r_0$ arbitrarily close to~$\ell/8$.
This yields $C_0 \simeq \frac{\log(g)}{g^{1/6}}$.
\end{proof}

In the hyperbolic case, one can also show that the minimum degree of the graph~$\Gamma$ introduced in the proof of Theorem~\ref{theo:ell} is at least three and that its maximum degree is bounded from above by a constant depending only on~$\ell$ (by the Bishop-Gromov inequality).
Therefore, using the result of~\cite{BV96} in the proof of Theorem~\ref{theo:ell} instead of the systolic inequality on graphs of~\cite{BT97,BS02}, one can prove the following result on the existence of short homologically independent disjoint cycles. 

\begin{theorem} 
On every closed hyperbolic orientable surface of genus~$g$ with homological systole at least~$\ell$, there exist at least $\displaystyle \left[C_{2} \, \frac{g}{\log(g+1)} \right]+1$ homologically independent disjoint loops~$\alpha_{k}$ such that
$$
\length(\alpha_{k}) \leq C_{3} \, \log (g+1),
$$
where the fonctions $C_2=C_2(\ell)$ and $C_3=C_3(\ell)$ only depend on~$\ell$.
\end{theorem}

\forgotten

\section{Short loops and the Jacobian of hyperbolic surfaces}

This section is dedicated to generalizing the results of P. Buser and P. Sarnak~\cite{BS94} in the following way. We begin by extending the $\log(g)$ upper bound on the length of the shortest homological non-trivial loop to almost $g$ loops, and then use these bounds and the methods developed in~\cite{BS94} to obtain information on the geometry of Jacobians. 

\subsection{Short homologically independent loops on hyperbolic surfaces}
We begin by showing that one can extend the usual $\log(g)$ bound on the homological systole of a hyperbolic surface, to a set of almost $g$ homologically independent loops. In the case where the homological systole is bounded below by a constant, this is a consequence of Theorem~\ref{theo:ell}, so here we show how to deal with surfaces with small curves.
More precisely, our result is the following:

\begin{theorem} \label{theo:lng}
Let $\eta:\N \to \N$ be a function such that 
$$
\displaystyle \lambda := \sup_g \frac{\eta(g)}{g} < 1.
$$
Then there exists a constant~$C_\lambda=\frac{2^{17}}{1-\lambda}$ such that for every closed hyperbolic orientable surface~$M$ of genus~$g$ there are at least $\eta(g)$ homologically independent loops $\alpha_{1},\ldots,\alpha_{\eta(g)}$ which satisfy
\begin{equation} \label{eq:upper1}
\length(\alpha_{i}) \leq C_\lambda \, \log(g+1)
\end{equation}
for~every $i \in \{1,\ldots,\eta(g)\}$ and admit a collar of width at least
$$
w_0={1\over 2}\arcsinh(1)
$$
around each of them.
\end{theorem}

The previous theorem applies with $\eta(g)=[\lambda g]$, where $\lambda \in (0,1)$.

\begin{remark}\label{rem:ln}
The non-orientable version of this theorem is the following: if $\eta:\N \to \N$ is a function such that 
$$
\displaystyle \lambda := \sup_g \frac{\eta(g)}{g} < \frac{1}{2},
$$
then there exists a constant~$c'$ such that for every closed hyperbolic non-orientable surface~$M$ of genus~$g$ (homeomorphic to the sum of $g$ copies of the projective plane) there are at least $\eta(g)$ homologically independent loops $\alpha_{1},\ldots,\alpha_{\eta(g)}$ which satisfy
$$
\length(\alpha_{i}) \leq \frac{C}{1-2\lambda} \, \log(g+1)
$$
for~every $i \in \{1,\ldots,\eta(g)\}$.
\end{remark}

The following lemma about minimal homology bases, \cf~Definition~\ref{def:min}, is proved in \cite[Section~5]{gro83}, see also~\cite[Lemma~2.2]{guth}. Note it applies to Riemannian surfaces, not only hyperbolic ones. 

\begin{lemma} \label{lem:simple}
Let $M$ be a compact Riemannian surface with geodesic boundary components.
Then every minimal homology basis~$(\gamma_i)_i$ of $M$ is formed of simple closed geodesics such that for every~$i$, the distance between every pair of points of~$\gamma_i$ is equal to the length of the shortest arc of~$\gamma_i$ between these two points.
Furthermore, two different loops of~$(\gamma_i)_i$ intersect each other at most once.
\end{lemma}

\begin{remark} \label{rem:straight}
The homotopical systolic loops of~$M$ also satisfy the conclusion of Lemma~\ref{lem:simple}.
\end{remark}

\begin{remark}\label{rem:diameter}
In particular, Lemma~\ref{lem:simple} tells us that one can always find a homology basis of length at most twice the diameter of the surface. In \cite{GPY}, 
this is shown to be false for lengths of pants decompositions. Indeed, a pants decomposition of a ``random" pants decomposition (where for instance random is taken in the sense of R.~Brooks and E.~Makover) there is a curve of length at least $g^{\frac{1}{6}-\varepsilon}$ and the diameter behaves roughly like $\sim \log(g)$. 
\end{remark}

In the next lemma, we construct particular one-holed hyperbolic tori, which we call {\it fat tori} for future reference. 

\begin{lemma}\label{lem:fat}
Let $\varepsilon \in (0,2 \arcsinh(1)]$.
There exists a hyperbolic one-holed torus~$T$ such that
\begin{enumerate}
\item the length of its boundary~$\partial T$ is equal to~$\varepsilon$;
\item the length of every homotopically nontrivial loop of~$T$ is at least~$\varepsilon$; \label{item22}
\item every arc with endpoints in~$\partial T$ representing a nontrivial class in~$\pi_{1}(T,\partial T)$ is longer than any homologically nontrivial loop of~$T$.
\end{enumerate}
\end{lemma}

\begin{proof}
We construct $T$ as follows. We begin by constructing, for any $\varepsilon \in (0,2 \arcsinh(1)]$, the unique right-angled pentagon $P$ with one side of length $\frac{\varepsilon}{4}$, and the two sides not adjacent to this side of equal length, say $a$. By the pentagon formula,
$$
a = \arcsinh\sqrt{\cosh\frac{\varepsilon}{4}}.
$$
A simple calculation shows that $2a> \varepsilon$ for $\varepsilon \leq 2 \arcsinh (1)$.

For future reference, denote by~$h$ the length of the two edges of~$P$ adjacent to the side of length~$\frac{\varepsilon}{4}$.
Now, we glue four copies of $P$, along the edges of length $h$ to obtain a square with a hole, \cf~Figure~\ref{fig0}.

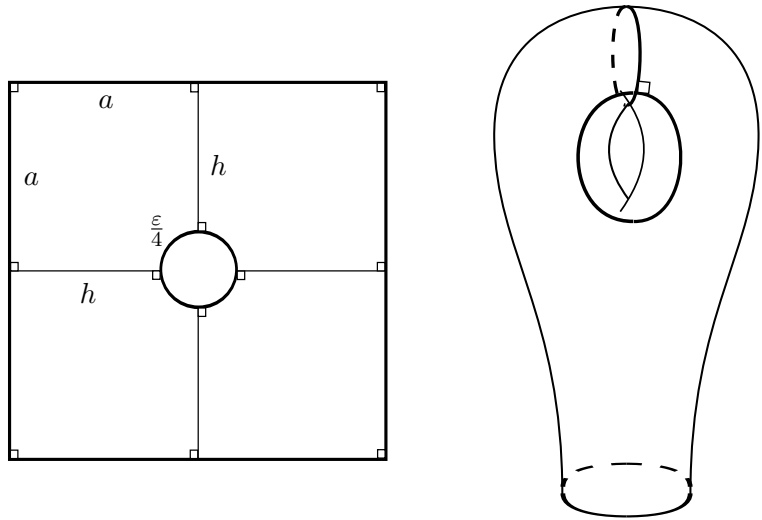
\begin{figure}[h]
\leavevmode 
\SetLabels
\L(.21*.65) $a$\\
\L(.27*.8) $a$\\
\L(.36*.67) $h$\\
\L(.255*.42) $h$\\
\L(.31*.55) $\frac{\varepsilon}{4}$\\
\endSetLabels
\begin{center}
\AffixLabels{\centerline{\epsfig{file =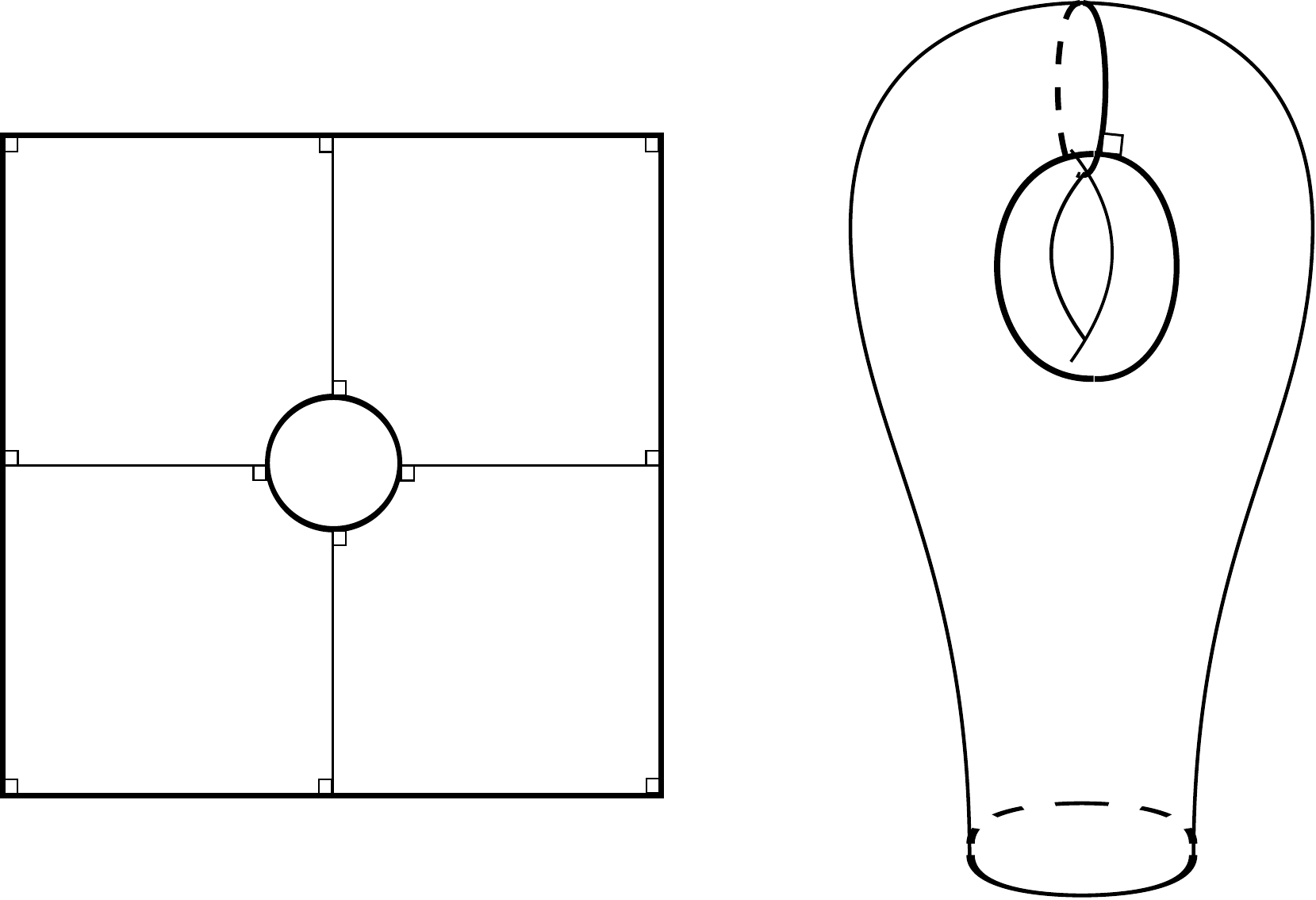,width=10cm,angle=0}}}
\end{center}
\caption{The construction of a fat torus} \label{fig0}
\end{figure}

We glue the opposite sides of this square to obtain a one-holed torus $T$ with boundary length $\varepsilon$. Note that the sides of the square project onto two simple closed geodesics $\gamma_1$ and~$\gamma_2$ of length~$2a$.
The distance between the two connected boundary components of $T \setminus \{\gamma_i\}$ arising from~$\gamma_i$ is at least~$2a$.
Therefore, as every noncontractible simple loop of~$T$ not homotopic to the boundary~$\partial T$ has a nonzero intersection number with $\gamma_1$ or~$\gamma_2$, its length is at least $2a>\varepsilon$.
We immediately deduce the point~\eqref{item22}.
This also shows that $\gamma_1$ and~$\gamma_2$ form a minimal homology basis of~$T$.
In particular, the homological systole of~$T$ is equal to~$2a$.

Now, in the one-holed square, the distance between the boundary circle of length~$\varepsilon$ and each of the sides of the square is clearly equal to~$h$.
Thus, the arcs of~$T$ with endpoints in~$\partial T$ homotopically nontrivial in the free homotopy class of arcs with endpoints lying on~$\partial T$ are of length at least~$2h$.
By the pentagon formula, 
$$
\sinh h \, \sinh \frac{\varepsilon}{2} = \cosh a.
$$
As $\varepsilon \leq 2 \arcsinh(1)$, we conclude that $2h > 2a$.
This proves the lemma.
\end{proof}



\forget
Let $T$ be the hyperbolic pair of pants with three boundary components of length~$\varepsilon$.
The three minimizing arcs, of length~$a$, joining the boundary components of the punctured torus decompose~$T$ into two isometric right-angled hexagons.
Each of these hexagons can be further decomposed into two isometric right-angled pentagons along a minimizing arc, of length~$b$, joining two opposite sides.
The lengths of the sides of these pentagons are $\varepsilon/4$, $a$, $\varepsilon/2$, $a/2$ and $b$ (in cyclic order).
By the pentagon formula, \cf~\cite{bus92}, we have
\begin{eqnarray*}
\cosh a & = & \coth(\varepsilon/4) \, \coth(\varepsilon/4) \\
\cosh b & = & \sinh(a) \, \sinh(\varepsilon/2).
\end{eqnarray*}
Since $\varepsilon \leq 2 \arcsinh(1)$, we derive, after some computation, that $\varepsilon \leq a \leq 2b$, which yields the desired result.
\forgotten

We can now proceed to the main proof of the section.

\begin{proof}[Proof of Theorem \ref{theo:lng}] \mbox{ }
 
\medskip

\noindent{\it Step 1.} 
We know this theorem to be true if the homological systole $\sys_H(M)$ is at least $2 \arcsinh(1)$ so suppose that this does not hold.\\

\noindent Consider $\alpha_1,\ldots,\alpha_{n}$ the set of homologically independent closed geodesics of $M$ of length less than $2 \arcsinh(1)$.
Note that by the collar lemma, these curves are simple and disjoint, and there is a collar of width~$1>w_0$ around each of them. \\

\noindent If $n\geq \eta(g)$ then the theorem is correct with $C_\lambda = 2 \arcsinh(1)/\log(2)$ so let us suppose that $n < \eta(g)$.
Let us consider $N= M \setminus \{\alpha_1,\ldots,\alpha_n\}$, a surface of signature $(g-n,2n)$.
The homological systole of~$N$ is at least $2 \arcsinh(1)$.
Using \cite{parlier}, we can now deform $N$ into a new hyperbolic surface~$N'$ which satisfies the following properties:

\begin{enumerate}
\item the boundary components of $N'$ are geodesic of length exactly $2 \arcsinh(1)$ in $N'$;
\item for any simple loop $\gamma$, we have
$$
\ell_{N'}(\gamma) \geq \ell_{N}(\gamma).
$$
In particular, 
$$
\sys_H(N') \geq \sys_H(N) \geq 2 \arcsinh(1).
$$
\end{enumerate}
Recall that $\ell_{N}(\gamma)$ is the length of the shortest loop of~$N$ homotopic to~$\gamma$. \\

\noindent We define a new surface $S$ by gluing $2n$ hyperbolic fat tori $T_1,\hdots,T_{2n}$ (from Lemma \ref{lem:fat} with $\varepsilon = 2 \arcsinh(1)$) to the boundary geodesics of $N'$.
This new surface is of genus $g+n$ and its homological systole is at least $2 \arcsinh(1)$.
As such, Theorem~\ref{theo:ell} implies that the first $\eta(g)+3n$ homologically independent curves $\gamma_1,\hdots,\gamma_{\eta(g)+3n}$ of a minimal homology basis of $S$ satisfy
\begin{eqnarray}
\ell_{S}(\gamma_k) & \leq & C_0 \frac{\log (2g+2)}{2g-\eta(g)-n} \, (g+n) \nonumber\\
& \leq & C_0 \, \frac{g}{g-\eta(g)} \, \log (2 g+2) \nonumber \\
& \leq & C_\lambda \, \log(g+1) \label{eq:Clambda}
\end{eqnarray}
for~$C_\lambda = \frac{2C_{0}}{1-\lambda} > 2 \arcsinh(1)/\log(2)$, from the assumption on~$\eta$ and since $n < \eta(g)$. \\

\noindent Although the surfaces $M$ and $S$ are possibly non homeomorphic, it should be clear what homotopy class we mean by $\alpha_k$ on $S$.
By construction of~$S$, the curves $\alpha_1,\ldots,\alpha_n$ are homologically trivial (separating) homotopical systolic loops of~$S$.
Furtermore, the loops $\gamma_k$ lie in a minimal homology basis of $S$.
Thus, from Lemma~\ref{lem:cut} and Remark~\ref{rem:straight}, the curves $\gamma_k$ are disjoint from the curves $\alpha_{1},\ldots,\alpha_n$.
In particular, $\gamma_{k}$ lies either in~$N'$ or in a fat torus~$T_{i}$. \\

\noindent Among the $\gamma_1,\hdots,\gamma_{\eta(g)+3n}$, some of them can lie in the fat tori $T_k$.
There are $2n$ fat tori and at most two homologically independent loops lie inside each torus.
Therefore, there are at least $\eta(g)-n$ curves among the $\gamma_k$ which lie in $N'$.
Renumbering the indices if necessary, we can assume that the loops $\gamma_{1},\ldots,\gamma_{\eta(g)-n}$ lie in~$N'$.
These $\eta(g)-n$ loops~$\gamma_{i}$ induce $\eta(g)-n$ loops in~$M$, still denoted by $\gamma_{i}$, through the inclusion of~$N'$ into~$M$.
Combined with the curves $\alpha_{1},\ldots,\alpha_{n}$ of~$M$, we obtain $\eta(g)$ loops $\alpha_{1}.\ldots,\alpha_{n},\gamma_{1},\ldots,\gamma_{\eta(g)-n}$ in~$M$. \\

\noindent Let us show that these loops are homologically independent in~$M$.
Consider an integral cycle
\begin{equation} \label{eq:cycle}
\sum_{i=1}^n a_i \, \alpha_i + \sum_{j=1}^{\eta(g)-n} b_j \, \gamma_j
\end{equation}
homologuous to zero in~$M$.
Since the curves $\alpha_i$ lie in the boundary $\partial N$ of $N$, the second sum of this cycle represents a trivial homology class in $H_1(N,\partial N;\Z)$, and so in $H_1(S;\Z)$.
Now, since the $\gamma_j$ are homologically independent in~$S$, this implies that all the $b_j$ equal zero.
Thus, the first sum in~\eqref{eq:cycle} is homologically trivial in~$M$.
As the $\alpha_i$ are homologically independent in~$M$, we conclude that all the $a_i$ equal zero too. \\

\noindent The curves $\alpha_k$ have their lengths bounded from above by $2 \arcsinh(1)$ and clearly satisfy the estimate~\eqref{eq:upper1}.
Now, since the simple curves $\gamma_k$ do not cross any of the $\alpha_i$, 
we have
$$
\ell_M(\gamma_k) = \ell_{N}(\gamma_{k}) \leq \ell_{N'}(\gamma_k) = \ell_{S}(\gamma_k).
$$
Therefore, the lengths of the $\eta(g)$ homologically independent geodesic loops $\alpha_{1}.\ldots,\alpha_{n},\gamma_{1},\ldots,\gamma_{\eta(g)-n}$ of~$M$ are bounded from above by $C_\lambda \, \log(g)$, with $C_{\lambda} = \frac{2^{17}}{1-\lambda}$. \\

\noindent{\it Step 2.} To complete the proof of the theorem, we need a lower bound on the width of the collars of these $\eta(g)$ closed geodesics in~$M$.
We already know that the curves $\alpha_1,\ldots,\alpha_n$ admit a collar of width~$w_0$ around each of them.
Without loss of generality, we can assume that the geodesics $\gamma_{1},\ldots,\gamma_{\eta(g)-n}$, which lie in~$N$, are part of a minimal basis of~$N$. \\

\noindent Let $\gamma$ be one of these simple closed geodesics. 
Recall that $\gamma$ is disjoint from the $\alpha_i$.
Suppose that the width~$w$ of its maximal collar is less than~$w_0$.
Then there exists a non-selfintersecting geodesic arc~$c$ of length~$2w$ intersecting $\gamma$ only at its endpoints.
Let $d_1$ be the shortest arc of~$\gamma$ connecting the endpoints of~$c$ and $d_2$ the longest one.
From Lemma~\ref{lem:simple} applied to~$N$, the arc~$d_1$ is no longer than~$c$.
Therefore, 
$$
\length(c \cup d_1) \leq 4w < 2 \, \arcsinh(1).
$$
By definition of the $\alpha_i$, the simple closed geodesic $\delta_1$ homotopic to the loop~$c \cup d_1$, of length less than $2 \arcsinh(1)$, is homologuous to an integral combination of the~$\alpha_i$.
Thus, the intersection number of $\gamma$ and~$\delta_1$ is zero.
As these two curves intersect at most once, they must be disjoint.
Denoting by~$\delta_2$ the simple closed geodesic homotopic to the loop~$c \cup d_2$, we deduce that $\gamma$, $\delta_1$ and $\delta_2$ bound a pair of pants. 
The curve~$c$ along with the three minimizing arcs joining $\gamma$ to $\delta_1$, $\gamma$ to $\delta_2$ and $\delta_1$ to~$\delta_2$ decompose this pair of pants into four right-angled pentagons.
From the pentagon formula,
$$
\cosh \left( \frac{1}{2} \length(\delta_i) \right) = \sinh(w) \, \sinh \left( \frac{1}{2} \length(d_i) \right)
$$
Since $w \leq \arcsinh(1)$, we deduce that $\delta_i$ is shorter than~$d_i$.
Hence,
$$
\length(\delta_1)+\length(\delta_2)\leq\length(\gamma).
$$
This yields another contradiction.
\end{proof}

\forget

As $\length(h) = w \leq \arcsinh(1)$,

This pair of pants is naturally divided into two isometric right-angled hexagons which in turn split into four right-angled pentagons obtained by cuting each hexagon along the height $h$ between $\gamma$ and its opposite side. As $\length(h) = w \leq \arcsinh(1)$, elementary hyperbolic trigonometry considerations in the right-angled pentagons involved lead to the inequality

\noindent{\it Step 2.} To complete the proof of the theorem, we need a lower bound on the width of the collars of these $\eta(g)$ closed geodesics in~$M$.
We already know that the curves $\alpha_1,\ldots,\alpha_n$ admit a collar of width~$w_0$ around each of them.
Without loss of generality, we can assume that the geodesics $\gamma_{1},\ldots,\gamma_{\eta(g)-n}$, which lie in~$N$, are part of a minimal basis of~$N$. \\

\noindent Let $\gamma$ be one of these geodesics. Suppose that the width~$w$ of its maximal collar is less than~$w_0$.
There exists a non-selfintersecting geodesic arc~$c$ of length~$2w$ intersecting $\gamma$ only at its endpoints. Let $d$ be the shortest arc of~$\gamma$ connecting the endpoints of~$c$.
From Lemma~\ref{lem:simple} applied to~$N$, the arc~$d$ is no longer than~$c$.
Therefore, 
$$
\length(c \cup d) \leq 4w < 2 \, \arcsinh(1).
$$
Hence, the simple loop~$c \cup d$ is homotopic to some noncontractible closed geodesic~$\alpha$ of length less than $2 \arcsinh(1)$. \\

\noindent Consider now a collar~$U$ of~$\alpha$ of width $w(\alpha)=\arcsinh(1/\sinh(\frac{1}{2} \length(\alpha)))$.
The length~$\ell$ of the boundary components of~$U$ is less than~$2\sqrt{2} \arcsinh(1)$, \cf~\cite[\S~4.1]{bus92}.
(Recall indeed that the length of~$\alpha$ is less than~$2 \arcsinh(1)$.)
Let $\delta$ be the distance from $c \cup d$ to~$M \setminus U$.
Remark that the injectivity radius of~$M$ at every point of~$c \cup d$ is less or equal to~$2w$.
Hence, from the collar theorem, \cf~\cite[Theorem~4.1.6]{bus92}, we derive
\begin{eqnarray*}
\sinh(2w) & \geq & \cosh \left( \frac{1}{2} \length(\alpha) \right) \, \cosh(\delta) - \sinh(\delta) \\
& \geq & \cosh(\delta) - \sinh(\delta) = e^{-\delta}.
\end{eqnarray*}
Thus, since $w < \frac{1}{2} \arcsinh(e^{-\sqrt{2} \arcsinh(1)}) \leq \frac{1}{2} \arcsinh(e^{-\ell/2})$, the loop~$c \cup d$ lies in~$U$ at distance greater than~$\ell/2$ from~$\partial U$.
In this case, the arc of~$\gamma \cap U$ containing~$d$ with endpoints in~$\partial U$, of length greater than~$\ell$, can be homotoped to a shorter arc of~$\partial U$ while keeping its endpoints fixed (recall that $\gamma$ does not intersect~$\alpha$).
This yields a contradiction because $\gamma$ is length minimizing in its homotopy class.
\forgotten

\subsection{Jacobians of Riemann surfaces} \label{sec:jacobian}

In this section, we present an application of the results of the previous section to the geometry of Jacobians of Riemann surfaces, extending the work~\cite{BS94} of P.~Buser and P.~Sarnak. \\

Consider a closed Riemann surface~$M$ of genus~$g$.
We define the $L^2$-norm $|.|_{L^2}$, simply noted $|.|$, on $H^1(M;\R) \simeq \R^{2g}$ by setting
\begin{equation} \label{eq:inf}
|\Omega|^2 = \inf_{\omega \in \Omega} \int_M \omega \wedge *\omega
\end{equation}
where $*$ is the Hodge star operator and the infimum is taken over all the closed one-forms~$\omega$ on~$M$ representing the cohomology class~$\Omega$.
The infimum in~\eqref{eq:inf} is attained by the unique closed harmonic one-form in the cohomology class~$\Omega$.
The space $H^1(M;\Z)$ of the closed one-forms on~$M$ with integral periods (that is, whose integrals over the cycles of~$M$ are integers) is a lattice of~$H^1(M;\R)$.
The Jacobian~$J$ of~$M$ is a principally polarized abelian variety isometric to the flat torus
$$
\TT^{2g} \simeq H^1(M;\R)/H^1(M;\Z)
$$
endowed with the metric induced by~$|.|$. \\

In their seminal article~\cite{BS94}, P.~Buser and P.~Sarnak show that the homological systole of the Jacobian of~$M$ is bounded from above by $\sqrt{\log(g)}$ up to a multiplicative constant.
In other words, there is a nonzero lattice vector in~$H^1(M;\Z)$ whose $L^2$-norm satisfies a $\sqrt{\log(g)}$ upper bound.
We extend their result by showing that there exist almost $g$ linearly independent lattice vectors whose norms satisfy a similar upper bound.
More precisely, we have the following.

\begin{theorem}
Let $\eta:\N \to \N$ be a function such that 
$$
\displaystyle \lambda := \sup_g \frac{\eta(g)}{g} < 1.
$$
Then there exists a constant~$C_\lambda=\frac{2^{17}}{1-\lambda}$ such that for every closed Riemann surface~$M$ of genus~$g$ there are at least $\eta(g)$ linearly independent lattice vectors $\Omega_{1},\ldots,\Omega_{\eta(g)} \in H^1(M;\Z)$ which satisfy
\begin{equation} \label{eq:upper2}
|\Omega_i|^2 \leq C_\lambda \, \log(g+1)
\end{equation}
for~every $i \in \{1,\ldots,\eta(g)\}$.
\end{theorem}

\begin{proof}
Let $\alpha_{1},\ldots,\alpha_{\eta(g)}$ be the homologically independent loops of~$M$ given by Theorem~\ref{theo:lng}.
Following~\cite{BS94}, for every~$i$, we consider a collar~$U_i$ of width $w_0=\frac{1}{2} \arcsinh(1)$ around~$\alpha_i$.
Let $F_i$ be a smooth function defined on~$U_i$ which takes the value $0$ on one connected component of~$\partial U_i$ and the value~$1$ on the other.
We define a one-form~$\omega_i$ on~$M$ with integral periods by setting $\omega_i=dF_i$ on~$U_i$ and $\omega_i=0$ ouside~$U_i$.
Let $\Omega_i$ be the cohomology class of~$\omega_i$, that is, $\Omega_i=[\omega_i]$.
Clearly, we have
$$
|\Omega_i|^2 \leq \inf_{F_i} \int_{U_i} dF_i \wedge *dF_i
$$
where the infimum is taken over all the functions~$F_i$ as above.
This infimum agrees with the capacity of the collar~$U_i$.
Now, by~\cite[(3.4)]{BS94}, we have
$$
|\Omega_i|^2 \leq \frac{\length(\alpha_i)}{\pi-2 \theta_0}
$$
where $\theta_0 = \arcsin(1/\cosh(w_0))$. 

Since the homology class of~$\alpha_i$ is the Poincar\'e dual of the cohomology class~$\Omega_i$ of~$\omega_i$, the cohomology classes~$\Omega_i$ are linearly independent.
The result follows from Theorem~\ref{theo:lng}.
\end{proof}

\section{Short homologically independent loops on Riemannian surfaces}

This section is devoted to the proof of the Riemannian version of Theorem \ref{theo:lng}.
More precisely, we show the following.

\begin{theorem} \label{theo:lngr}
Let $\eta:\N \to \N$ be a function such that 
$$
\displaystyle \lambda := \sup_g \frac{\eta(g)}{g} < 1.
$$
Then there exists a constant~$C_\lambda=\frac{2^{18}}{1-\lambda}$ such that for every closed orientable Riemannian surface~$M$ of genus~$g$ there are at least $\eta(g)$ homologically independent loops $\alpha_{1},\ldots,\alpha_{\eta(g)}$ which satisfy
\begin{equation} \label{eq:upperr}
\length(\alpha_{i}) \leq C_\lambda \, \frac{\log(g+1)}{\sqrt{g}} \, \sqrt{\area(M)}
\end{equation}
for every $i \in \{1,\ldots,\eta(g) \}$.
\end{theorem}

\begin{remark}
In \cite[Theorem 5.3.E]{gro83}, M.~Gromov also proposes an extension of a systolic 
inequality to families of curves. Note however that neither the growth 
rate (roughly $g^{1-\varepsilon}$ for any positive~$\varepsilon$) nor the number of curves is close to being 
optimal. Furthermore, even an adaptation of the proof taking into 
account the stronger systolic inequality \cite[Theorem 2.6]{gro96} gives a bound of order $\sqrt{g^{c \log g}}.\log g$ for the $g/4$ first curves, where $c$ is some positive constant. 
\end{remark}

\begin{remark}
A non-orientable version of this statement also holds, \cf~Remark~\ref{rem:ln}.
\end{remark}

\begin{proof}[Proof of Theorem~\ref{theo:lngr}]
Multiplying the metric by a constant if necessary, we can assume that the area of~$M$ is equal to~$g$. 
The strategy of the proof is to deform the surface~$M$ into another surface to which we can apply Theorem~\ref{theo:ell}.
Through this deformation, we want to make the homological systole uniformly bounded from below while controlling the area and the length of some minimal homology basis. \\

\noindent{\it Step 1.} 
Fix $\varepsilon=1$.
Consider a maximal collection of simple closed geodesics $\alpha_{1},\ldots,\alpha_{n}$ of length at most~$\varepsilon$ which are pairwise disjoint and homologically independent in~$M$.
We have $n \leq g$.
Let us cut the surface open along the loops~$\alpha_{i}$.
We obtain a compact surface~$N$ of signature $(g-n,2n)$.
Each loop~$\alpha_i$ gives rise to two boundary components $\alpha_i^+$ and~$\alpha_i^-$.
By construction, the simple closed geodesics of~$N$ of length at most~$\varepsilon$ are separating.

If $n \geq \eta(g)$ then the theorem clearly holds.
We will therefore assume that $n < \eta(g)$. \\

\noindent{\it Step 2.} 
Set $\alpha=\alpha_1^+$.
Consider a loop~$c$ homotopic to~$\alpha$ of length at most~$\varepsilon$ on~$N$ which bounds a domain~$C$ of maximal area with~$\alpha$.
By the Ascoli theorem, such a loop exists.
The domain~$C$ may not be homeomorphic to a cylinder because $\alpha$ may have (nontransverse) self-intersections.
However, if we remove~$C$ from~$N$ and glue a cylinder along~$c$, we get a surface homeomorphic to~$N$.

The domain~$C$ of~$N$ does not cover the whole surface~$N$.
Otherwise, every $1$-dimensional homology class of~$N$ could be represented by a loop lying in the support of~$c$.
Since some $1$-dimensional homology classes of~$N$ cannot be represented by the sum of loops of length at most~$\varepsilon$, the length of~$c$ would be greater than~$\varepsilon$, which is impossible.

Thus, the loop $c$ decomposes~$N$ into two domains with nonempty interiors: $C$ and its complementary.
This implies that the length of~$c$ is equal to~$\varepsilon$, otherwise we could push the loop~$c$ away from~$\alpha$ in a neighborhood of some point of~$c$.

Now, we replace the domain~$C$ of~$N$ with the flat cylinder~$S_{\varepsilon} \times [0,\varepsilon/2]$, where $S_{\varepsilon}$ is the circle of length~$\varepsilon$.
The resulting surface~$N_1$ has the same signature as~$N$ and its area is less or equal to $\area(M)+\varepsilon^2/2$.
Furthermore, no loop of~$N_1$ of length less than~$\varepsilon$ is homotopic to the boundary component of~$N_1$ corresponding to~$\alpha$, otherwise we could derive a contradiction with the definition of~$c$.
In addition, the marked length spectrum of~$N_{1}$ is clearly greater or equal to the one of~$N$, \cf~Definition~\ref{def:mls}. \\

\noindent We successively repeat this process with the remaining $2n-1$ curves $\alpha_{1}^-,\alpha_2^{\pm},\ldots,\alpha_{n}^{\pm}$.
At the end, we obtain a surface~$N'=N_{2n}$, homeomorphic to~$N$, of area less or equal to $\area(N)+n \varepsilon^{2}$ whose marked length spectrum is greater or equal to the one of~$N$.
By construction, the loops of~$N'$ of length less than~$\varepsilon$ are separating and nonhomotopic to the boundary components of~$N'$. \\

\noindent{\it Step 3.}
We can now follow the proof of Theorem~\ref{theo:lng}.
As the $2n$ boundary components of~$N'$ are isometric to~$S_{\varepsilon}$, we can attach to each of them a fat torus, \cf~Lemma~\ref{lem:fat}.
This yields a genus $g+n$ surface~$S$ with homological systole at least~$\varepsilon$.
From Theorem~\ref{theo:ell} applied to~$S$, the first $\eta(g)+3n$ homologically independent loops $\gamma_{1},\ldots,\gamma_{\eta(g)+3n}$ of~$S$ satisfy the following inequality (compare with~\eqref{eq:Clambda})
\begin{equation}
\length(\gamma_{i}) \leq 2 \, C_\lambda \, \log(g+1) \label{eq:logg+2}
\end{equation}
for $C_{\lambda}=\frac{2C_{0}}{1-\lambda}$, where $C_{0}=2^{17}$, since $n < \eta(g) < g$ and $\area(S) = g+2n \, 2 \pi + n \varepsilon^2 \leq (2+4\pi) g$. \\

\noindent Now, from Lemma~\ref{lem:cut}, the loops~$\gamma_{i}$ do not cross the curves~$\alpha_k$ and therefore lie either in a fat torus or in $N'$.
Arguing as in the proof of Theorem~\ref{theo:lng}, we obtain $\eta(g)-n$ loops among the~$\gamma_{i}$, say $\gamma_1,\ldots,\gamma_{\eta(g)-n}$, which form with $\alpha_1,\ldots,\alpha_n$ a family of $\eta(g)$ homologically independent loops of~$M$.

The lengths of the $\alpha_k$ are bounded from above by~$\varepsilon$ and satisfy the upper bound~\eqref{eq:upperr}.
Since the curves $\gamma_k$ do not cross any of the $\alpha_i$, we also have
$$
\ell_M(\gamma_k) = \ell_N(\gamma_k) \leq \ell_{N'}(\gamma_k) = \ell_{S}(\gamma_k).
$$
Therefore, the lengths of the $\eta(g)$ loops $\alpha_{1},\ldots,\alpha_n, \gamma_{1}.\ldots,\gamma_{\eta(g)-n}$ are bounded from above by $C_\lambda \, \log(g)$, with $C_{\lambda} = \frac{2^{18}}{1-\lambda}$.
\end{proof}

Theorem~\ref{theo:lngr} also leads to the following upper bound on the sum of the lengths of the $g$ shortest homologically independent loops of a genus~$g$ Riemannian surface.

\begin{corollary}
There exists a universal constant~$C$ such that for every closed Riemannian surface~$M$ of genus~$g$, the sum of the lengths of the $g$ shortest homologically independent loops is bounded from above by
$$
C \, g^{3/4} \sqrt{\log(g)} \, \sqrt{\area(M)}.
$$
\end{corollary}

\begin{proof}
Let $\lambda \in (0,1)$.
From Theorem~\ref{theo:lngr}, the lengths of the first $[\lambda g]$ homologically independent loops are bounded by $\displaystyle \frac{C'}{1-\lambda} \, \frac{\log(g+1)}{\sqrt{g}} \, \sqrt{\area(M)}$, while from~\cite[1.2.D']{gro83}, the lengths of the next $g-[\lambda g]$ others are bounded by $\displaystyle C'' \, \sqrt{\area(M)}$, for some universal constants $C'$ and $C''$.
Thus, the sum of the $g$ shortest homologically independent loops is bounded from above by
$$
C' \, \frac{[\lambda g]}{1-\lambda} \, \frac{\log(g+1)}{\sqrt{g}} \, \sqrt{\area(M)} + C'' \, (g-[\lambda g]) \, \sqrt{\area(M)}.
$$
The result follows from a suitable choice of $\lambda$.
\end{proof}

\forget
We can assume that the loops $\alpha_{1},\ldots,\alpha_{\eta(g)+3k}$ are length-minimizing representatives of a minimal homology basis on~$S$.
As in Lemma~\ref{lem:disj}, the loops~$\alpha_{i}$ lie either in a fat torus or in $N'_{m'} \setminus \{\gamma'_{1},\ldots,\gamma'_{m'}\}$.
Since at most two homologically independent loops lie in each fat torus, there are at least $\eta(g)-k$ curves among the~$\alpha_{i}$ which lie in $N'_{m'} \setminus \{\gamma'_{1},\ldots,\gamma'_{m'}\}$.
Furthermore, we can assume that these $\eta(g)-k$ loops are homologically independent in $H_{1}(N'_{m'},\partial N'_{m'}) \simeq H_{1}(N,\partial N)$.

From Steps 2 and 3, the homotopy classes of these $\eta(g)-k$ curves can be represented in $N' \simeq N$, and so in $M$, by $\eta(g)-k$ loops satisfying the relation~\eqref{eq:logg+2}.
Since these $\eta(g)-k$ loops are homologically independent with the $k$ loops $\gamma_{1},\ldots,\gamma_{k}$ of~$M$ of length at most~$\varepsilon$, we obtain the desired result.
\forgotten

\forget
Let $(M,\G_0)$ be a closed Riemannian surface of genus $g$ with area less than $g$ and systole less than $\epsilon$. 

Let $c$ be a closed geodesic homologically non trivial with length less than $\epsilon$.
We cut $M$ along $c$ and obtain a new surface $M'$ with two boundary geodesics $\alpha$ and $\beta$.
Let $C^\epsilon_\alpha$ denote the completion of the set consisting of all simple loops of $M'$ in the homotopy class of $\alpha$ with length less than $\epsilon$.
This set is not empty as $\alpha \in C^\epsilon_\alpha$ and we can define a partial order on it as follows.
Two elements $\alpha_1$ and $\alpha_2$ of $C^\epsilon_\alpha$ satisfies the relation $\alpha_1 \leq \alpha_2$ if the cylinder of $M'$ bounded by both $\alpha$ and $\alpha_2$ contained $\alpha_1$.
So $(C^\epsilon_\alpha,\leq)$ is a partially ordered set.
Furthermore every chain $L$, {\it i.e.} a totally ordered subset, of $C^\epsilon_\alpha$ admits an upperbound.
For this consider the union of all the cylinder bounded by both $\alpha$ and an element of $L$.
This is a cylinder whose boundary is an element of $C^\epsilon_\alpha$. 
By Zorn's lemma the set $C^\epsilon_\alpha$ admits a maximal element $\alpha'$.
We remove the cylinder bounded by $\alpha$ and $\alpha'$ and define $C^\epsilon_\beta$ in the new surface obtained.
We can also find a maximal element $\beta'$ in $C^\epsilon_\beta$.
By removing the cylinder bounded by $\beta$ and $\beta'$, we obtain a surface denoted by $M''$.
Either the two boundary curves $\alpha'$ and $\beta'$ have length $\epsilon$, or their union is a graph $\Gamma$ of length less than $2\epsilon$ such that 
$$
p : \pi_1(\Gamma) \to \pi_1(M'',\partial M'')
$$
is injective. In this last case we can find $[g/2]$ homogically independent curves of length less than $\epsilon$ and the proof is finished.
If they do not coincide, then $\alpha'$ and $\beta'$ have length $\epsilon$. Then we glue a cylinder of width $\epsilon$ on $M''$ by identifying $\alpha'$ with one side of the cylinder and $\beta'$ with the other side. We then obtain a surface denoted by $M_1$ with
$$
\area(M_1) \leq \area(M)+\epsilon^2;
$$
such that every loop in the homotopy of $c$ or intersecting the homotopy class of $c$ has length at least $\epsilon$.
\forgotten

\section{Asymptotically optimal hyperbolic surfaces} \label{sec:cex}

In this section, we show that the bound obtained in Theorem~\ref{theo:lngr} on the number of homologically independent loops of length at most $\sim \log(g)$ on a genus~$g$ hyperbolic surface is optimal.
Namely, we construct hyperbolic surfaces whose number of homologically independent loops of length at most $\sim \log(g)$ does not grow asymptotically like~$g$.
Specifically, we prove the following (we refer to Definition~\ref{def:ksys} for the definition of the $k$-th homological systole).

\begin{theorem} \label{theo:cex}
Let $\eta:\N \to \N$ be a function such that $\displaystyle \lim_{g \to \infty} \frac{\eta(g)}{g} = 1$.
Then there exists a sequence of genus~$g_{k}$ hyperbolic surfaces~$M_{g_{k}}$ with $g_{k}$ tending to infinity such that
$$
\lim_{k \to \infty} \frac{\sys_{\eta(g_{k})}(M_{g_{k}})}{\log(g_{k})} = \infty.
$$
\end{theorem}

\medskip

Before proceeding to the proof of this theorem, we will need the following constructions. \\

All hyperbolic polygons will be geodesics and convex.
We will say that a hyperbolic polygon is symmetric if it admits an axial symmetry. \\

Fix $\ell > 0$ such that $\cosh(\ell) > 7$ and let $L,L'>0$ be large enough (to be determined later). \\

\noindent{\it Construction of a symmetric hexagon.}
With our choice of~$\ell$, there is no hyperbolic triangle with a basis of length~$\ell$ making an angle greater or equal to~$\pi/6$ with the other two sides.
Therefore, there exists a symmetric hyperbolic hexagon $H_{\pi/6,L}$ (resp. $H_{\pi/3,L}$) with a basis of length~$\ell$ forming an angle of~$\pi/6$ (resp.~$\pi/3$) with its adjacent sides such that all its other angles are right and the side opposite to the basis is of length~$L$, \cf~Figure~\ref{fig2}.
Note that the length of the two sides adjacent to the side of length~$L$ goes to zero when $L$ tends to infinity. \\

\begin{figure}[h]
\leavevmode \SetLabels
\L(.5*.78) $L$\\
\L(.5*.09) $\ell$\\
\endSetLabels
\begin{center}
\AffixLabels{\centerline{\epsfig{file =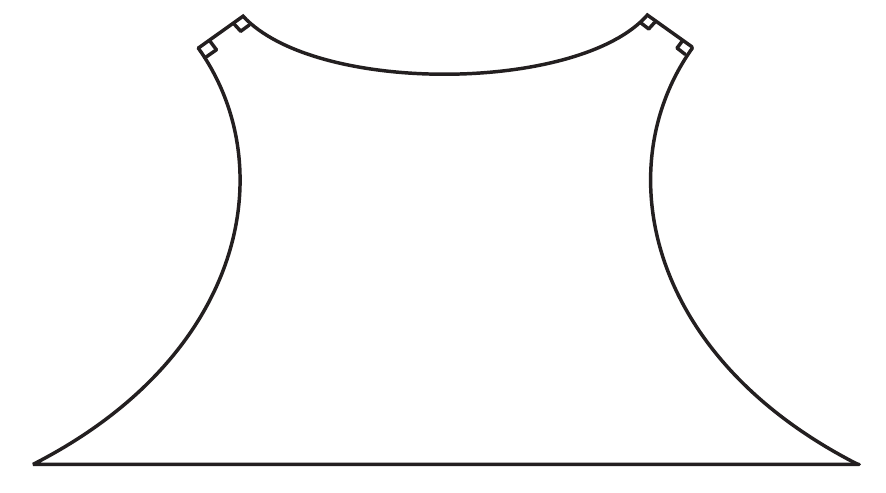,width=10cm,angle=0}}}
\end{center}
\caption{The schematics for $H_{\pi/6,L}$ and $H_{\pi/3,L}$}\label{fig2}
\end{figure}

\noindent{\it Construction of a $3$-holed triangle.}
Consider the hyperbolic right-angled hexagon~$H_{L}$ with three non-adjacent sides of length~$L$.
Note that the lengths of the other three sides go to zero when $L$ tends to infinity.
By gluing three copies of~$H_{\pi/6,L}$ to~$H_{L}$ along the sides of length~$L$, we obtain a three-holed triangle~$X_{3}$ with angles measuring~$\pi/3$ and sides of length~$\ell$, where the three geodesic boundary components can be made arbitrarily short by taking $L$ large enough, \cf~Figure~\ref{fig3}.
We will assume that the geodesic boundary components of~$X_{3}$ are short enough to ensure that the widths of theirs collars are greater than~$e^{g}$, with $g$ to be determined later. \\

\begin{figure}[h]
\leavevmode \SetLabels
\L(.445*.4) $L$\\
\L(.53*.4) $L$\\
\L(.49*.23) $L$\\
\L(.377*.48) $\ell$\\
\L(.5*.07) $\ell$\\
\L(.61*.48) $\ell$\\
\endSetLabels
\begin{center}
\AffixLabels{\centerline{\epsfig{file =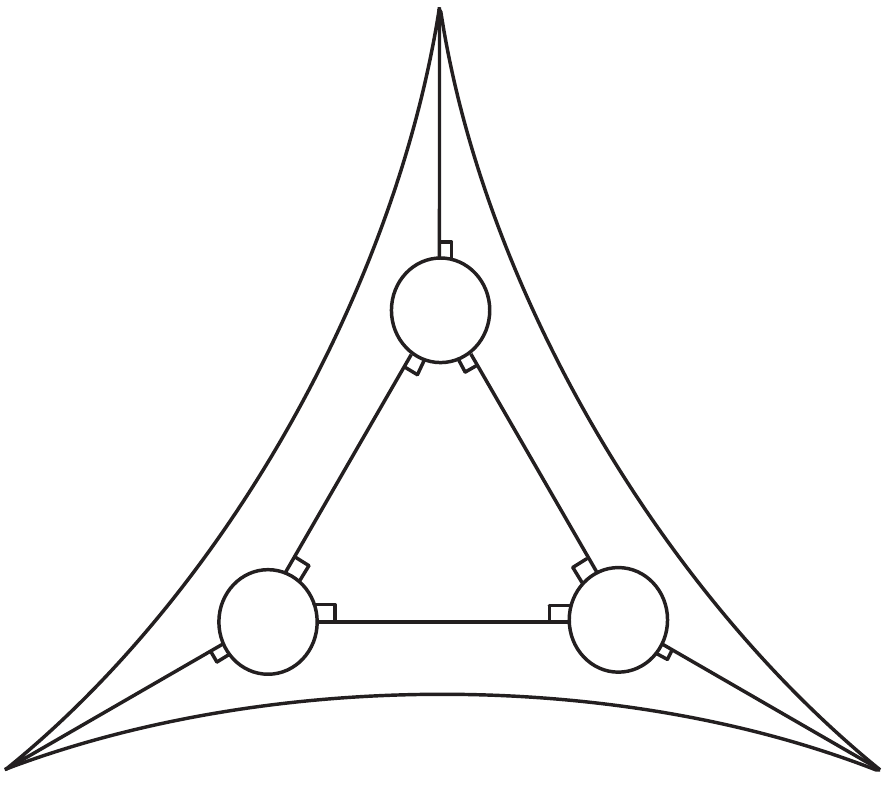,width=10cm,angle=0}}}
\end{center}
\caption{} \label{fig3}
\end{figure}

\noindent{\it Construction of a $7$-holed heptagon.} There exists a symmetric hyperbolic pentagon~$P_{L'}$ with all its angles right except for one measuring~$2\pi/7$ such that the length of the side opposite to the non-right angle is equal to~$L'$.
As previously, the length of the two sides adjacent to the side of length~$L'$ goes to zero when $L'$ tends to infinity.
By gluing seven copies of~$P_{L'}$ around their vertex with a non-right angle, we obtain a hyperbolic right-angled $14$-gon, \cf~Figure~\ref{fig4}.
Now, we paste along the sides of length~$L'$ seven copies of~$H_{\pi/3,L'}$.
We obtain a $7$-holed heptagon~$X_{7}$ with angles measuring~$2\pi/3$ and sides of length~$\ell$, where the seven geodesic boundary components can be made arbitrarily short by taking $L'$ large enough, \cf~Figure~\ref{fig4}.
We will assume that the geodesic boundary components of~$X_{7}$ are as short as the geodesic boundary components of~$X_{3}$, which guarantees that the widths of theirs collars are also greater than~$e^{g}$, with $g$ to be determined later. \\

\begin{figure}[h]
\leavevmode \SetLabels
\L(.42*.88) $\ell$\\
\L(.44*.74) $L'$\\
\L(.485*.25) $P_{L'}$\\
\endSetLabels
\begin{center}
\AffixLabels{\centerline{\epsfig{file =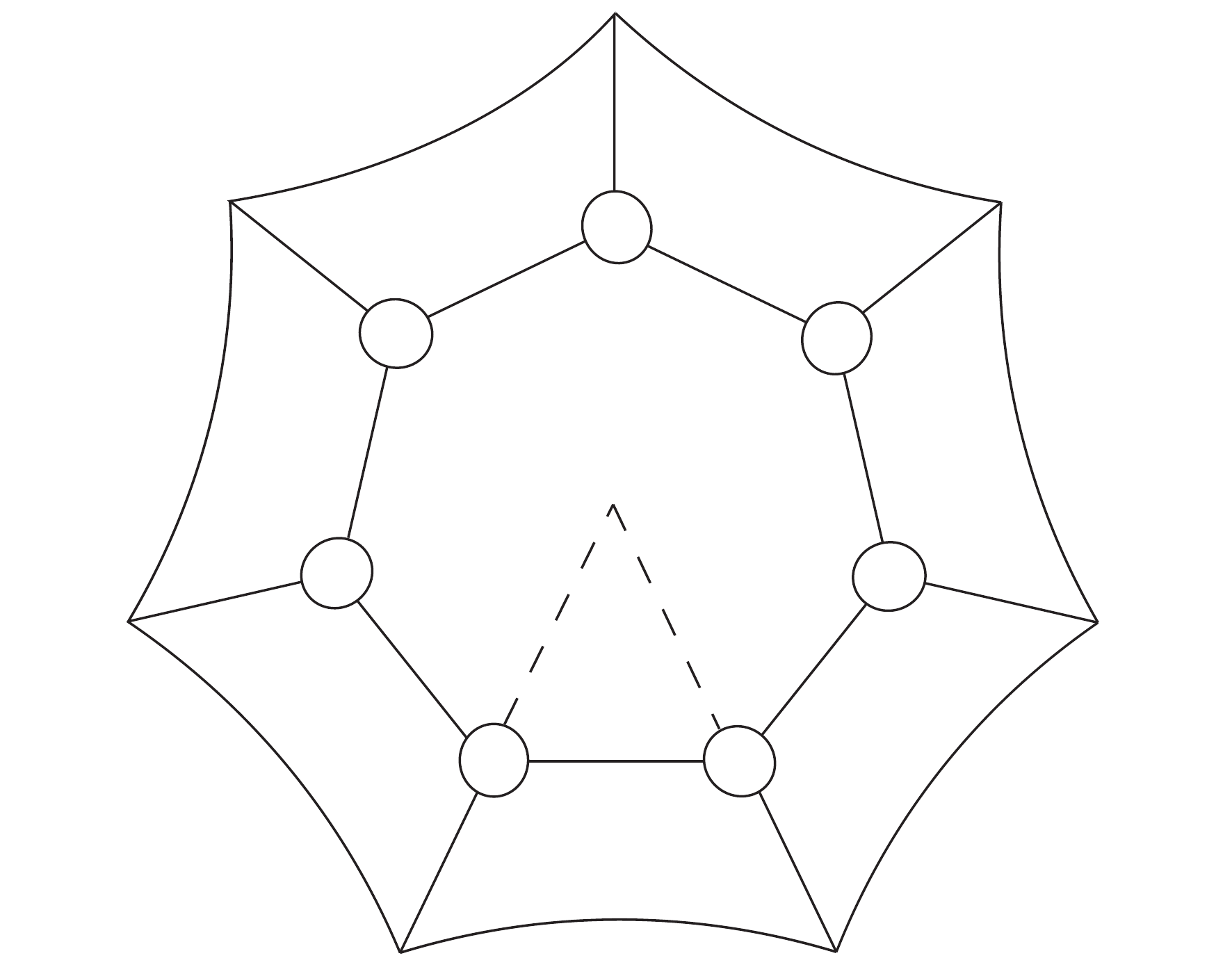,width=8cm,angle=0}}}
\end{center}
\caption{The $7$-holed heptagon} \label{fig4}
\end{figure}

\noindent{\it Hurwitz surfaces with large systole.}
Generalizing the work of Buser-Sarnak~\cite{BS94}, Katz-Schaps-Vishne~\cite{KSV07} constructed a family of hyperbolic surfaces~$N$ defined as a principal congruence tower of Hurwitz surfaces with growing genus~$h$ and with homological systole at least $\frac{4}{3} \log(h)$.
Since Hurwitz surfaces are $(2,3,7)$-triangle surfaces, they admit a triangulation~$\T$ made of copies of the hyperbolic equilateral triangle with angles~$2\pi/7$.
The area of this triangle equals~$2\pi/7$.
Therefore, the triangulation~$\T$ of~$N$ is formed of $14(h-1)$ triangles and has $6(h-1)$ vertices.
Remark that not every integer~$h$ can be attained as the genus of a Hurwitz surface with homological systole at least~$\frac{4}{3} \log(h)$.
Still, this is true for infinitely many $h$'s. \\

\noindent{\it Adding handles.}
In order to describe our construction, it is more convenient to replace the previous hyperbolic equilateral triangles of~$\T$ with Euclidean equilateral triangles of unit side length, which gives rise to a piecewise flat surface $N_0 \simeq N$.
For every of these Euclidean triangles~$\Delta$, consider a subdivision of each of its sides into $m$ segments of equal length.
The lines of~$\Delta$ parallel to the sides of the triangle and passing through the points of the subdivision decompose~$\Delta$ into~$m^{2}$ Euclidean equilateral triangles of size~$\frac{1}{m}$.
These small triangles define a new (finer) triangulation~$\T'$ of~$N$ with exactly seven triangles around each vertex of the original triangulation~$\T$ and exactly six triangles around the new vertices.
Note that the new triangulation~$\T'$ is formed of $14(h-1)m^2$ triangles.
Now, replace each heptagon (with a conical singularity in its center) formed of the seven small triangles of~$\T'$ around the vertices of the original triangulation~$\T$ by a copy of the hyperbolic $7$-holed heptagon~$X_{7}$ (of side length~$\ell$).
Replace also the other small triangles of~$\T'$ by a copy of the hyperbolic three-holed triangle~$X_{3}$ (of side length~$\ell$).
The conditions on the angles of $X_{3}$ and~$X_{7}$ imply that the resulting surface~$M'$ is a compact hyperbolic surface with geodesic boundary components, of signature~$(h,42(h-1)(m^{2}-2))$.
Note that the lengths of the nonboundary closed geodesics of~$M'$ are bounded away from zero by a positive constant $\kappa=\kappa(\ell)$ depending only on~$\ell$.
By gluing the geodesic boundary components pairwise, which all have the same length, we obtain a closed hyperbolic surface~$M$ of genus~$g=h+21(h-1)(m^{2}-2)$.

\begin{remark} \label{rem:punc}
It is not possible to use single-punctured hyperbolic polygons in our construction and still have the required conditions on their angles and the lengths of their sides to obtain a smooth closed hyperbolic surface at the end.
Note also that the combinatorial structure of Hurwitz surfaces, and more generally triangle surfaces, makes the description of our surfaces simple.
\end{remark}

The following lemma features the main property of our surfaces.

\begin{lemma} \label{lem:logh}
Let $k=21(h-1)(m^{2}-2)$.
The $(k+1)$-th homological systole of~$M$ is large.
More precisely, there exists a universal constant~$K \in (0,1)$ such that
$$
\sys_{k+1}(M) \geq \frac{4}{3} \, Km \, \log(h).
$$
\end{lemma}

\begin{proof}
Let us start with some distance estimates.
No matter how small the geodesic boundary components of~$X_3$ are, the distance between two points of its non-geodesic boundary component~$\partial_0 X_3$ is greater or equal to $K$ times the distance between the corresponding two points in the boundary~$\partial \Delta$ of a Euclidean triangle of unit side length.
That is,
$$
{\rm dist}_{\partial_0 X_3 \times \partial_0 X_3} \geq K \, {\rm dist}_{\partial \Delta \times \partial \Delta},
$$
where $K$ is a universal constant.
If, instead of a Euclidean triangle of unit side length, we consider a small Euclidean triangle of size~$\frac{1}{m}$, we have to change $K$ into~$Km$ in the previous bound.
The same inequality holds, albeit with a different value of~$K$, if one switches $X_3$ for~$X_7$.
Here, of course, we should replace the small Euclidean triangle with the singular Euclidean heptagon that $X_7$ replaces in the construction of~$M$. \\

Now, let us estimate the $(k+1)$-th homological systole of~$N$.
By construction, there are $k$ short disjoint closed geodesics $\alpha_1,\ldots,\alpha_k$ of the same length which admit a collar in~$M$ of width at least~$e^g$.
Furthermore, the loops~$\alpha_i$ are homologically independent in~$M$.

Let $\gamma$ be a geodesic loop in~$M$ homologically independent from the~$\alpha_i$.
We can suppose that $\gamma$ does not intersect the loops~$\alpha_i$, otherwise its length would be at least~$e^g$ and we would be done.
Thus, if the trajectory of~$\gamma$ enters into a copy of $X_3$ or~$X_7$ (through a side of length~$\ell$), it will leave it through a side of length~$\ell$.
Therefore, using the previous distance estimates, the curve~$\gamma$ induces a homotopically nontrivial loop in~$N_0 \simeq N$ of length at most
$$
\frac{1}{Km} \, \length(\gamma).
$$
Since $N$ and $N_0$ are bilipschitz equivalent, it does not matter with respect to which metric we measure the lengths; the only effect might be on the constant~$K$.
Now, by construction the homological systole of~$N$ is at least~$\frac{4}{3} \log(h)$, we conclude that
$$
\length(\gamma) \geq \frac{4}{3} \, Km \, \log(h).
$$
\end{proof}

With this contruction, we can prove Theorem~\ref{theo:cex}.

\begin{proof}[Proof of Theorem~\ref{theo:cex}]
Given $\eta$ as in the theorem, let us show that for every $C>1$, there exist infinitely many hyperbolic surfaces~$M_{g}$ with pairwise distinct genus~$g$ such that
\begin{equation} \label{eq:C}
\frac{\sys_{\eta(g)}(M_{g})}{\log(g)} \geq C.
\end{equation}
The surfaces~$M_g$ have already been constructed in the discussion following Theorem~\ref{theo:cex}.
We will simply set the parameters $h$ and $m$ on which they depend so that the inequality~\eqref{eq:C} holds. \\

Let $\varepsilon > 0$ such that $\varepsilon \leq \frac{1}{100(\frac{3C}{2K}+1)^{2}}$ and $m = E(\frac{1}{10 \sqrt{\varepsilon}})$.
Note that $\varepsilon \leq \frac{1}{100}$, $m \geq \frac{3C}{2K}$ and $\varepsilon m^2 \leq \frac{1}{100}$.
By assumption on~$\eta$, there exists an integer~$g_{0} \geq 100$ such that for every $g \geq g_{0}$, we have $\eta(g) \geq (1-\varepsilon)g$.
Now, we can set $h \geq \max \{ 21(m^{2}-2),g_{0} \}$ for which there exists a genus~$h$ Hurwitz surface~$N$ with homological systole at least $\frac{4}{3} \log(h)$.
Remark that there are infinitely many choices for~$h$.

From the construction following Theorem~\ref{theo:cex}, we obtain infinitely many hyperbolic surfaces~$M_{g}$ with pairwise distinct genus~$g=h+21(h-1)(m^{2}-2)$.
Now, from our choice of parameters, we have 
$$
\eta(g) \geq (1-\varepsilon) g \geq 21(h-1)(m^{2}-2)+1.
$$
Combined with Lemma~\ref{lem:logh}, this implies that 
$$
\sys_{\eta(g)}(M_{g}) \geq \frac{4}{3} \, Km \, \log(h).
$$
From the expression of~$g$ and our choice of parameters, we derive
$$
\log(g) \leq \log(h) + \log(21(m^{2}-2)) \leq 2 \, \log(h) \leq \frac{4}{3} \, \frac{Km}{C} \, \log(h).
$$
Therefore, 
$$
\sys_{\eta(g)}(M_{g}) \geq C \, \log(g).
$$
Hence the theorem.
\end{proof}

\begin{remark}
Note that the surfaces constructed in the previous theorem are hyperbolic, not merely Riemannian.
The construction of {\it nonhyperbolic} Riemannian surfaces satisfying the conclusion of Theorem~\ref{theo:cex} is easier to carry out as there is no constraint imposed by the hyperbolic structure anymore.
We can start with a Buser-Sarnak sequence of surfaces and attach long thin handles to these surfaces.
Then we can argue as in the proof of the theorem.
\end{remark}

\forget
\begin{proposition} 
For any $C>0$, the quantity
\begin{equation}\label{eqn:contra}
C \log(g) \sqrt{\frac{\area(S)}{g}}
\end{equation}
cannot be a universal upper bound on the lengths of $f(g)$ homologically distinct curves where $f$ is a function of $g$ with $\lim_{g\to \infty} \frac{f(g)}{g}=1$.
\end{proposition}

\begin{proof}

Suppose by contradiction that this was the case, i.e., that there exists a bound given by equation \ref{eqn:contra} on the length of $f(g)$ homologically independent curves with $\lim_{g\to \infty} \frac{f(g)}{g}=1$. In particular, this implies that for any $\varepsilon>0$, and sufficienly high $g$, equation \ref{eqn:contra} holds for at least $ag=(1-\varepsilon)g$ curves. 

We begin with a Buser-Sarnak sequence of surfaces. More precisely, a family of hyperbolic surfaces with growing genus $h$ and with systole growth $\frac{4}{3}\log(h)$. We begin by cutting out $2k$ disjoint small disks of circumference $\varepsilon$ (for some very small $\varepsilon$). We group the $\varepsilon$ boundary curves into pairs, and we attach $k$ thin Euclidean cylinders of boundary length $\varepsilon$ and of length $L$. We choose $L$ so that any curve that goes through a cylinder is extremely long, say length $e^h$. We choose $\varepsilon$ so that the total area of the cylinders is less than some constant, say $4\pi$. As such, the surface obtained has total area at most $4\pi(h-1) + 4\pi=4\pi h$, and genus $g=h+k$.

Notice that there are $k$ short homotopically and homologically distinct curves on the surface $\sigma_1,\hdots,\sigma_k$, those given by the parallel curves to the cylinders, and the remaining curves must be of length greater or equal to the systole of the original surface, i.e., of length 
at least $\frac{4}{3}\log(h) = \frac{4}{3}\log(g-k)$. As such we have for any curve $\gamma$ homologically independent from $\sigma_1,\hdots,\sigma_k$:
$$
\ell(\gamma)\geq \frac{4}{3}\log(g-k).
$$
Because of what precedes, we know we need to consider $k =(1- \lambda) g$ for $\lambda \in ]0,\frac{1}{2}[$. Thus
$$
\ell(\gamma)\geq \frac{4}{3}\log(\lambda g).
$$
Consider $\lambda$ such that $\lambda< \frac{4}{9\pi C^2}$ (we should be thinking of large $C$ so this quantity is less than $1$). Thus:
$$
\frac{4}{3}-2\sqrt{\pi}C \sqrt{\lambda}>0
$$
and for sufficiently large $g$, we have
$$
\left(\frac{4}{3}-2\sqrt{\pi}C \sqrt{\lambda}\right)\log(g)>\frac{4}{3}\log(\lambda^{-1}).
$$
Now, as the area of our surface is at most $4\pi h = 4\pi \lambda g$, we can conclude that
$$
\ell(\gamma)\geq \frac{4}{3}\log(\lambda g)>2\sqrt{\pi}C \log(g) \sqrt{\frac{\lambda g}{g}}=C \log(g) \sqrt{\frac{4\pi\lambda g}{g}}>C\log(g) \sqrt{\frac{\area(S)}{g}}.
$$
Remark that if the inequality holds for $a g$ curves then 
$$
a\leq 1- \frac{4}{9\pi C^2}.
$$
\end{proof}

Remark that we have not found a contradiction to the possibility that for any $a<1$, there exists a $C$ sufficiently large such that the desired inequality would hold for $a g$ curves. Also remark that the above examples are not hyperbolic. By adapting the hairy torus argument to ``hairy" Buser-Sarnak surfaces, the above examples can be mirrored by hyperbolic surfaces, but they aren't as easy to construct, and do not achieve better results.
\forgotten

\section{Pants decompositions} \label{sec:pants1}

In this section, we establish bounds on the length of short pants decompositions of genus~$g$ Riemannian surfaces with $n$ marked points. Specifically, we establish bounds using two different measures of length. Classically, one measures the length of a pants decomposition by considering the length of the longest curve in the pants decomposition. This is the point of view taken by P.~Buser in his quantification of Bers' constants for instance. One can also measure the total length of a pants decomposition. We establish news bounds for both measures but our primary goal is to establish bounds on sums of lengths of pants decompositions for surfaces with marked points and hyperelliptic surfaces.

\subsection{Some preliminary bounds}

Recall that the {\it Bers constant} of a Riemannian surface~$M$, denoted by~$\BB(M)$, is the length of a shortest pants decomposition of~$M$, where the length of a pants decomposition~$\PP$ of $M$ is defined as
\begin{equation} \label{eq:pants1}
\length(\PP)=\max_{\gamma \in \PP} \, \length(\gamma).
\end{equation}

In Section \ref{sec:pants2}, we will also consider the sum of the lengths of a pants decomposition~$\PP$ (or total length of~$\PP$) defined as 
\begin{equation} \label{eq:pants2}
\sum_{\gamma \in \PP} \length(\gamma).
\end{equation}

We begin by observing that the two subjects of finding homologically non-trivial short curves, and finding short pants decompositions, are related.

\begin{lemma}\label{lem:bershomology}
Every pants decomposition of a genus $g$ surface contains $g$ homologically independent disjoint loops.
\end{lemma}

\begin{proof}
Consider a pants decomposition. Clearly it contains a curve such that once removed, the genus is $g-1$. The remaining curves form a pants decomposition of the surface with the curve removed, so they contain a curve such that once removed the genus is $g-2$, and so forth until the remaining genus is $0$. The $g$ curves are clearly homologically distinct. 
\end{proof}

One could try to find a short pants decomposition by considering disjoint homologically non-trivial loops, and then completing them into a pants decomposition. And indeed, in the case where there are $g$ homologically independent disjoint loops of length at most~$C \, \log(g)$, we can derive a near-optimal bound on the Bers constant of the surface.

\begin{proposition} \label{theo:ell'''}
Let $M$ be a closed orientable hyperbolic surface of genus~$g$ which admits $g$ homologically nontrivial disjoint loops of length at most $C \, \log(g)$.
Then
$$
\BB(M)\leq C' \, \sqrt{g} \, \log(g)
$$
for $C'=46 \, C$. In particular such surfaces satisfy 
$$
\BB(M)\leq C' \, g^{\frac{1}{2}+\varepsilon}
$$
for any positive $\varepsilon$ and large enough genus. 
\end{proposition}

Recall that there exist examples of closed hyperbolic surfaces of genus $g$ with Bers' constant at least $\sqrt{6g}-2$, and that it is conjectured that it cannot be substantially improved (see \cite{bus92}).

\begin{proof}
Cut $M$ open along the $g$ homologically nontrivial disjoint loops of length at most $C \, \log(g)$.
The resulting surface $M'$ is a sphere of area $4\pi (g-1)$ with $2g$ boundary components of length at most $C \, \log(g)$.
By~\cite[Theorem~5]{BP09}, there exists a pants decomposition of $M'$ of length bounded by 
$$
46\, \sqrt{2\pi(2g-2)}\sqrt{\left(\frac{C\, \log(g)}{2\pi}\right)^2+1} \leq 46 \, C \, \sqrt{g} \, \log(g).
$$
The same result can also be derived from~\cite{BS10}, albeit with a worse multiplicative constant.
\end{proof}

Unfortunately, in light of the examples of families of surfaces with Cheeger constant uniformly away from zero, \cf~\cite{BR99}, one cannot hope for $\log(g)$ bounds on the length of too many disjoint loops in general.
So this strategy would need to be adapted to say anything new about the lengths of pants decompositions in general. 

\

We now derive some results which we will need in the next section to estimate the sum of the lengths of short pants decompositions, \cf~Theorem~\ref{thm:sumgenusg}. Our first estimate relies on the diastolic inequality for surfaces, \cf~\cite{BS10}.

\begin{proposition}\label{prop:bers}
Let $M$ be a closed Riemannian surface of genus $g$ with $n$ marked points.
Then $M$ admits a pants decomposition with respect to the marked points of length at most
$$
C \, \sqrt{g} \, \sqrt{\area(M)}
$$
for an explicit universal constant~$C$.
\end{proposition}

\begin{proof}
Let $f:M \to \R$ be a topological Morse function.
We suppose that there is only one critical point on each critical level set and that the marked points of~$M$ are regular points lying on different level sets.
Such an assumption is generically satisfied. \\

\noindent The function~$f$ factors through the Reeb space of~$f$ defined as the quotient
$$
G = {\rm Reeb}(f) = M / \sim,
$$
where the equivalence $x \sim y$ holds if and only if $f(x)=f(y)$ and $x$ and $y$ lie in the same connected component of the level set $f^{-1}(f(x))$.
More precisely, we have
$$
f = j \circ \overline{f}
$$
where $\overline{f}: M \to G$ and $j:G \to \R$ are the natural factor maps induced by $f$ and the equivalence relation~$\sim$.
Since $f$ is a topological Morse function, its Reeb space is a finite graph and the factor map $\overline{f} : M \to G$ is a trivial $S^1$-bundle over the interior of each edge of~$G$. \\

\noindent Now, we subdivide~$G$ so that $\overline{f}$ takes the marked points of~$M$ to vertices of~$G$.
From Morse theory, the disjoint simple loops formed by the preimages of the midpoints of the edges of~$G$ decompose~$M$ into pants, disks and cylinders.
Therefore, there exists a pants decomposition of~$M$ with respect to the marked points of length at most
$$
\sup_{t \in G} \length \, \overline{f}^{-1}(t)
$$
and, in particular, of length at most
\begin{equation} \label{eq:ft}
\sup_{t \in \R} \length \, f^{-1}(t).
\end{equation}

\noindent On the other hand, from~\cite{BS10}, there exists a function $f:M \to \R$ as above with
\begin{equation} \label{eq:bs}
\sup_{t \in \R} \length \, f^{-1}(t) \leq C \, \sqrt{g} \, \sqrt{\area(M)}
\end{equation}
for an explicit universal constant~$C$.
Strictly speaking, the main result of~\cite{BS10} is stated differently (using the notion of diastole), however the proof leads to the above estimate. \\

\noindent Combining the estimates~\eqref{eq:ft} and \eqref{eq:bs}, we derive the desired bound.
\end{proof}

\subsection{Sums of lengths of pants decompositions} \label{sec:pants2}

One of our motivations is the study of hyperbolic surfaces with cusps, or possibly cone points. Our techniques allow us to treat the more general case of Riemannian surfaces where the appropriate replacement for cusps are marked points (see the proof of Corollary \ref{coro:sumgenusg} to relate the two notions). We thus require the following extension of the notion of systole.

\forget
\begin{definition} \label{def:ms}
The \emph{marked homotopical systole} of a compact Riemannian surface~$M$ with boundary and $n$ marked points $x_1,\ldots,x_n$ is defined as the supremum of the reals~$\ell$ such that every simple loop of~$M' = M \setminus \{x_1,\ldots,x_n\}$ of length less than~$\ell$ is homotopic to a point in $M'$, a connected component of~$\partial M$ or a small circle around some marked point.
\end{definition}
\forgotten

\begin{definition} \label{def:ms}
Let $M$ be a compact Riemannian surface with (possibly empty) boundary and $n$ marked points $x_1,\ldots,x_n$.
A loop of~$M$ is \emph{admissible} if it lies in $M' = M \setminus \{x_1,\ldots,x_n\}$ and is not homotopic to a point in $M'$, a connected component of~$\partial M$ or a multiple of some small circle around some marked point~$x_i$.
The \emph{marked homotopical systole} of~$M$ is the infimal length of the admissible loop of~$M$.
\end{definition}

We will implicitly assume that the topology of~$M$ is such that admissible loops exist, otherwise there is nothing to prove.
Let us first establish the following result similar to Lemma~\ref{lem:regmet}.

\begin{lemma} \label{lem:regmetm}
Let $M_{0}$ be a closed Riemannian surface with $n$ marked points and marked homotopical systole greater or equal to~$\ell$.
Fix $0<R \leq \frac{\ell}{4}$.
Then there exists a closed Riemannian surface $M$ conformal to~$M_0$ such that
\begin{enumerate}
\item the area of~$M$ is less or equal to the area of~$M_{0}$;
\item $M \setminus \{x_1,\ldots,x_n\}$ and $M_0 \setminus \{x_1,\ldots,x_n\}$ have the same marked length spectrum; \label{eq:reg2}
\item the area of every disk of radius $R$ in $M$ is greater or equal to~$R^{2}/2$.
\end{enumerate}
\end{lemma}

The proof of this lemma closely follows the arguments of~\cite[5.5.C']{gro83} based on the height function.
In our case, we will need a modified version of it.

\begin{definition}
The \emph{tension} of an admissible loop~$\gamma$ of~$M$, denoted by $\tens(\gamma)$, is defined as 
$$
\tens(\gamma) = \length(\gamma) - \ell_{M'}(\gamma).
$$
Recall that $\ell_{M'}(\gamma)$ is the infimal length of the loops of~$M'$ freely homotopic to~$\gamma$.
The \emph{modified height function} of~$M$ is defined for every $x \in M'$ as the infimal tension of the admissible loops of~$M$ based at~$x$. It is denoted by~$h'(x)$.
\end{definition}

The following estimate is a slight extension of~\cite[5.1.B]{gro83}.

\begin{lemma} \label{lem:lowerbound}
Let $M$ be a closed Riemannian surface with $n$ marked points and marked homotopical systole greater or equal to~$\ell$.
Then
$$
\area \, D(x,R) \geq \left( R-\frac{h'(x)}{2} \right)^2
$$
for every $x \in M$ and every $R$ such that $\frac{1}{2} h'(x) \leq R \leq \frac{\ell}{4}$.
\end{lemma}

\begin{proof}
We argue as in~\cite[5.1.B]{gro83}.
Note first that the assumption clearly implies that the marked points are at distance at least~$\ell/2$ from each other.
If $R < \ell/4$, the disk $D=D(x,R)$ is contractible in~$M$ and contains at most one marked point of~$M$.
Thus, every admissible loop~$\gamma$ based at~$x$ contains an arc~$\alpha$ passing through~$x$ with endpoints in~$\partial D$.
This arc~$\alpha$ can be deformed into an arc~$\alpha'$ of~$\partial D$ in~$M'$ while keeping its endpoints fixed (indeed, $D$ contains at most one marked point).
The tension of~$\gamma$ satisfies
\begin{eqnarray*}
\tens(\gamma) & \geq & \length(\alpha) - \length(\alpha') \\
& \geq & 2R - \length(\partial D).
\end{eqnarray*}
Hence, $\length(\partial D) \geq 2R - h'(x)$.
By the coarea formula, we derive the desired lower bound on the area of~$D$ by integrating the previous inequality from $\frac{1}{2}h'(x)$ to~$R$.
\end{proof}

Our first preliminary result, namely Lemma~\ref{lem:regmetm}, can be derived from this estimate.

\begin{proof}[Proof of Lemma~\ref{lem:regmetm}]
The proof of~\cite[5.5.C'']{gro83} (see~\cite[Lemma~4.2]{RS08} for further details) applies with the modified height function and shows that there exists a closed Riemannian surface~$M$ conformal to~$M_0$ with a conformal factor less or equal to~$1$ which satisfies~\eqref{eq:reg2} and has its modified height function arbitrarily small.
The result follows from Lemma~\ref{lem:lowerbound}.
\end{proof}

Now, we can state the following estimate.

\begin{proposition}\label{prop:S}
Let $S$ be a Riemannian sphere with $n$ marked points $x_1,\ldots,x_n$.
Suppose that the marked homotopical systole of~$S$ is greater or equal to~$\ell$.
Then the sphere~$S$ admits a pants decomposition with respect to its marked points whose sum of the lengths is bounded from above by
$$
C\, \log(n) \frac{\area \, S}{\ell},
$$
where $C=2^{10}$.
\end{proposition}

\begin{proof}
We can suppose that $n \geq 4$, otherwise there is nothing to prove.
As noticed before, the assumption implies that the marked points $x_1,\ldots,x_n$ are at distance at least $\ell/2$ from each other.
By Lemma~\ref{lem:regmetm}, we can further assume that the area of every $r_{0}$-disk on~$S$ is greater or equal to~$r_{0}^{2}/2$, with $r_{0}=\ell/4$

\begin{lemma}\label{lem:path}
There exists a nonselfintersecting path~$\Gamma$ on~$S$ connecting all the marked points~$x_i$ with
\begin{equation} \label{eq:G}
\length \, \Gamma \, \leq C' \, \frac{\area \, S}{\ell}
\end{equation}
where $C'=2^{7}$.
\end{lemma}

\begin{proof}[Proof of Lemma \ref{lem:path}]
We consider the family of disjoint $r_{0}$-disks $D_1,\ldots,D_n$ centered at the marked points $x_1,\ldots,x_n$ of~$S$ with $r_0=\ell/4$.
We complete this family into a maximal family of disjoint $r_{0}$-disks~$\{D_i\}_{i \in I}$.
Since the area of these disks is greater or equal to~$r_{0}^{2}/2$, we have
\begin{equation} \label{eq:II}
 |I|\leq 2 \, \frac{\area \, S}{r_{0}^{2}}.
\end{equation}

\bigskip
 
\noindent We decompose the sphere~$S$ into Voronoi cells $V_i=\{p \in S \mid d(p,x_i)\leq d(p,x_j) \text{ for all } j\neq i\}$ around the points~$x_i$, with $i \in I$.
Each Voronoi cell is a polygon centered at some point~$x_{i}$.
Remark that a pair of adjacent Voronoi cells (\ie, meeting along an edge) corresponds to a pair of centers of distance at most~$4 r_0$.
To see this, consider a point~$p$ on the boundary of both of two adjacent cells.
It is at an equal distance~$\delta$ to both cell centers and it is closer to their centers than to any others.
Now $\delta \leq 2 r_0$, otherwise there would exist a disk of radius~$r_0$ around $p$ in~$S$ disjoint from all other disks~$\{D_i\}_{i \in I}$ and the system of disks would not be maximal.\\

\noindent We connect the center of every Voronoi cell~$V_{i}$ to the midpoints of its edges through length-minimizing arcs of~$V_{i}$.
(The length-minimizing arc connecting a pair of points is not necessarily unique, but we choose one.)
As a result, we obtain a connected embedded graph~$G$ in~$S$ with vertices~$\{x_{i}\}_{i \in I}$.
We already noticed that the lengths of the edges of~$G$ are at most~$4 r_0$. \\

\noindent The graph $\cup_i \partial V_i$ of~$S$ given by the Voronoi cell decomposition has the same number of edges~$e$ as~$G$.
Since the number of vertices in the Voronoi cell decomposition is less or equal to~$\frac{2e}{3}$, the Euler characteristic formula shows that the number of edges in the Voronoi cell decomposition and in~$G$ is at most $3|I|-6$.
Thus,
\begin{equation} \label{eq:lG}
\length \, G \leq 12 \, (|I|-2) \, r_0.
\end{equation}

\medskip

\noindent By considering the boundary~$\partial U$ of a small enough $\rho$-tubular neighborhood~$U$ of the minimal spanning tree~$T$ of~$G$, we obtain a loop surrounding~$T$ of length less than 
\begin{equation} \label{eq:2T}
2 \, \length \, T + \varepsilon
\end{equation}
for any given $\varepsilon >0$.
We then construct a nonselfintersecting path~$\Gamma$ connecting all the marked points~$x_1,\ldots,x_{n}$ with 
$$
\length \, \Gamma \leq 2 \, \length \, T + \varepsilon + 2n \rho.
$$
It suffices to take~$\Gamma$ lying in~$\partial U$ and modify it in the neighborhood of each point~$x_i$ by connecting~$\Gamma$ to~$x_i$ through two rays arising from~$x_{i}$. \\

\noindent The result follows from \eqref{eq:lG} and~\eqref{eq:II} by taking $\varepsilon$ and $\rho$ small enough.
\end{proof}

\medskip

\noindent 
Let us resume the proof of Proposition~\ref{prop:S}.
Let $\Gamma \subset S$ be as in the previous lemma.
Without loss of generality, we can suppose that $\Gamma$ is a piecewise geodesic path connecting the marked points $x_1,\ldots,x_n$ in this order. \\

\noindent We split the piecewise geodesic path $\Gamma=(x_1,\ldots,x_n)$ into two paths $\Gamma_1=(x_1,\ldots,x_m)$ and $\Gamma_2=(x_{m+1},\ldots,x_n)$ with $m=n/2$ if $n$ is even, and $m=(n+1)/2$ otherwise.
Now, we consider a loop~$\gamma$ surrounding~$\Gamma$ in~$S$, that is, the boundary of a small tubular neighborhood~$U$ of~$\Gamma$ in~$S$.
We also consider two loops $\gamma_1$ and~$\gamma_2$ surrounding $\Gamma_1$ and $\Gamma_2$ in~$U$.
Then we repeat this process with $\Gamma_1$ and $\Gamma_2$, and so forth, \cf~Figure~\ref{fig1}.
(When a path is reduced to a single marked point, we cannot split it any further.
So we take the same loop surrounding this marked point for $\gamma$, $\gamma_1$ and~$\gamma_2$.)
We stop the process at the step 
$$
\kappa= [\log_2 n] +1
$$
because after this step all the new loops surround a single marked point. \\

\begin{figure}[h]
\begin{center}
\ovalbox{\ovalbox{\ovalbox{$\bullet \quad \bullet$} \ovalbox{$\bullet \quad \bullet$}} \ovalbox{\ovalbox{$\bullet \quad \bullet$} $\bullet$}}
\end{center}
\caption{} \label{fig1}
\end{figure}



\noindent It is clear that from our construction some subfamily of the set of loops that arise gives a (nongeodesic) pants decomposition of the sphere~$S$ with respect to its marked points. \\

\noindent Since each segment of~$\Gamma$ is surrounded by at most~$\kappa$ loops, the sum of the lengths of the pants decomposition loops is bounded from above by
\begin{equation} \label{eq:kappa}
2 \kappa \, \length \, \Gamma + \varepsilon,
\end{equation}
where $\varepsilon$ depends on the width of~$U$ and goes to zero with it.
The desired upper bound follows from~\eqref{eq:G} by letting $\varepsilon$ go to zero.
\end{proof}

We now state our main theorem in this direction which states that for any fixed genus $g$, one can control the growth rate of sums of lengths of pants decompositions of a surface of area $\sim g+n$ by a factor which grows like $n \log(n)$, where $n$ is the number of marked points.

\begin{theorem}\label{thm:sumgenusg}
Let $M$ be a closed Riemannian surface of genus $g$ with $n \geq 1$ marked points whose area is equal to $4\pi(g+\frac{n}{2}-1)$.
Then $M$ admits a pants decomposition with respect to the marked points whose sum of the lengths is bounded from above by
$$
C_g\, n \log (n+1),
$$
where $C_g$ is an explicit genus dependent constant.
\end{theorem}

Before proving the theorem, we note the following corollary. 
\begin{corollary} \label{coro:sumgenusg}
Let $M$ be a noncompact hyperbolic surface of genus $g$ with $n$ cusps.
Then $M$ admits a pants decomposition with the sum of its lengths bounded above by
$$
C_g\, n \log (n+1),
$$
where $C_g$ is an explicit genus dependent constant.
\end{corollary}

\begin{proof}[Proof of Corollary~\ref{coro:sumgenusg}]
Let us cut off the cusps of the hyperbolic surface~$M$ and replace the tips with small round hemispheres to obtain a closed Riemannian surface~$N$ with $n$ marked points corresponding to the summits of the hemispheres.
The area of~$N$ can be made arbitrarily close to the original area by choosing the cut off tips of area arbitrarily small.
The hemispheres will then be of arbitrarily small area as well.
To avoid burdening the argument by epsilontics, we will assume that $M$ and $N$ have the same area.

We remark that short simple closed geodesics (on either surface) do not approach the tips.
Indeed, every simple loop passing through a small enough tip can be deformed into a shorter loop.
Therefore, the geodesic behavior that we are concerned with is identical on both surfaces.

We conclude by applying Theorem~\ref{thm:sumgenusg} to~$N$, which yields the desired pants decomposition on~$M$.
\end{proof}

\begin{remark}
The proof of Corollary~\ref{coro:sumgenusg} applies to noncompact complete Riemannian surfaces of genus~$g$ with $n$ ends whose area is normalized at $4\pi(g+\frac{n}{2}-1)$.
\end{remark}

Now, we can proceed to the proof of Theorem~\ref{thm:sumgenusg}, which relies on Propositions~\ref{prop:bers} and~\ref{prop:S}.

\begin{proof}[Proof of Theorem~\ref{thm:sumgenusg}]

\forget
Arguing as in the beginning of the proof of Theorem~\ref{theo:lng}, we can assume that there is no (simple) closed geodesic of length at most~$\ell$ on~$M$ [to develop]. \\

Let us cut off the cusps of the hyperbolic surface~$M$ and replace the tips with small round hemispheres to obtain a closed Riemannian surface~$N$.
The area of~$N$ can be made arbitrarily close to the original area by choosing the cut off tips of area arbitrarily small, and the hemispheres will then be of arbitrarily small area as well.
To avoid burdening the argument by epsilontics, we will assume that $M$ and $N$ have the same area. \\

Remark that short simple closed geodesics (on either surface) do not approach the tips.
Indeed, every simple loop passing through a small enough tip can be deformed into a shorter loop.
Therefore, the geodesic behavior that we are concerned with is identical on both surfaces.
Furthermore, we can assume that the tips are at distance at least~$\ell$ from any simple closed geodesic and from any other tips (on either surface). \\
\forgotten

We will prove a more general result.
Namely, Theorem~\ref{thm:sumgenusg} holds true for compact Riemannian surfaces~$M$ of signature $(g;p,q)$ (\ie, of genus $g$ with $p$ marked points $x_1,\ldots,x_p$ and $q$ boundary components) with boundary components of length at most~$\ell$, where $\ell:=1$.
In this case, $n=p+q$ represents the total number of marked points and boundary components of~$M$. \\

\noindent It is enough to show this result when the marked homotopical systole of $M$ is greater or equal to~$\ell$, \cf~Definition~\ref{def:ms}, otherwise we split the surface along a simple loop of $M \setminus \{ x_1,\ldots,x_p \}$ of length less than~$\ell$ nonhomotopic to a point, a boundary component or a small circle around a marked point.
Then we deal with the resulting surfaces.
Indeed, by splitting the surface~$M$ of signature $(g;p,q)$, we obtain one of the following:
\begin{enumerate}
\item a surface of signature $(g-1;p,q+2)$;
\item two surfaces of signature $(g_i;p_i,q_i)$ with $0<g_i < g$ and $p_i + q_i \leq p + q +2$ for $i=1,2$;
\item two surfaces of signature $(g_i;p_i,q_i)$ with $g_1=0$, $p_1 + q_1 \leq p + q +1$, $g_2=g$ and $p_2 + q_2 < p + q$;
\item or, in case $g=0$, two surfaces of signature $(0;p_i,q_i)$ with $p_i + q_i < p + q$ for $i=1,2$.
\end{enumerate}
In all cases, we can conclude by induction on both $g$ and $n=p+q$. \\

\noindent First, we attach a cap (\ie, a round hemisphere) along each boundary component $c_1,\ldots,c_q$ of~$M$ to obtain a closed Riemannian surface $N$ with $n=p+q$ marked points corresponding to the $p$ marked points of~$M$ and the $q$ summits of the caps of~$N$.
As with~$M$, the marked homotopical systole of $N$ is greater or equal to $\ell$.
By construction, we have
\begin{eqnarray*}
\area \, N & \leq & \area \, M + \sum_{i=1}^q \frac{1}{2\pi} \length(c_i)^2 \\
& \leq & 4 \pi \left( g+\frac{n}{2}-1 \right) + \frac{q}{2 \pi} \, \ell^2
\end{eqnarray*}
where $n=p+q$. \\

\noindent Now, we consider $g$ homologically independent disjoint loops $\gamma_1,\hdots,\gamma_g$ in a pants decomposition of~$N$ where the marked points are ignored with
$$
\length \,\gamma_{k} \leq A_g \, \sqrt{n},
$$
for every $k=1,\ldots,g$, where $A_g$ is some constant depending only on~$g$, \cf~Proposition \ref{prop:bers} and Lemma~\ref{lem:bershomology}.
Sliding these curves through a length-nonincreasing deformation if necessary, we can assume that they do not cross the caps of~$N$.
We can also assume that they do not pass through the marked points of~$N$ by slightly perturbing them and changing the value of~$A_g$. \\

\noindent We cut $N$ along $\gamma_1,\hdots,\gamma_g$ and attach caps along the $2g$ boundary components thus obtained.
As a result, we obtain a Riemannian sphere~$S$ with $2g+n$ marked points corresponding to the $n$ marked points of $N$ and the $2g$ summits of the caps we glued on~$N$.
By construction, the marked homotopical systole of~$S$ is greater or equal to~$\ell$ and 
$$
\area \, S \leq \area \, N + 2g \, \frac{A_{g}^{2} \, n}{2 \pi} \leq B_{g} \, n
$$
for some constant~$B_{g}$. \\

\noindent By applying Proposition~\ref{prop:S} to~$S$, we obtain a pants decomposition~$\mathcal{P}_{S}$ of~$S$ with respect to its marked points whose total length does not exceed
$$
\frac{C}{\ell} \, B_{g} \, n \, \log(2g+n).
$$

\noindent As previously, we can push a curve away from the caps of $S$ without increasing its length.
Therefore, we can assume that the pants decomposition loops of~$\mathcal{P}_{S}$ stay away from the caps.
This shows that the loops of~$\mathcal{P}_{S}$ also lie in~$M$, and form with $\gamma_1,\hdots,\gamma_g$ and the connected components of $\partial M$ a pants decomposition of~$M$.
By construction, the total length of this pants decomposition of~$M$ is bounded from above by
$$
\frac{C}{\ell} \, B_{g} \, n \, \log(2g+n) + g \, A_g \, \sqrt{n} + q \, \ell \leq C_{g} \, n \, \log (n+1)
$$
for some constant~$C_{g}$.
\end{proof}

We conclude this section with a corollary of the above result for hyperelliptic Riemannian surfaces, i.e., surfaces with an orientation-preserving isometric involution where the quotient surface is a sphere.

\begin{theorem} Every closed hyperelliptic Riemannian surface of genus $g$ and area $4\pi (g - 1)$ admits a pants decomposition whose sum of the lengths is bounded above by 
$$
C \, g \log (g) 
$$
for a universal constant $C$. 
\end{theorem}

\begin{proof}
We begin by taking the quotient of the hyperelliptic surface $M$ by its hyperelliptic involution $\sigma$ to obtain a sphere $S$ with $2g+2$ cone points of angle $\pi$. We denote these points $x_1,\hdots,x_{2g+2}$. \\

\noindent Using Theorem \ref{thm:sumgenusg}, there exists a pants decomposition of $S$ with marked points $x_1,\hdots,x_{2g+2}$ of total length which does not exceed $C_0 (2g+2) \log(2g+2)$ for some universal constant $C_0$. We proceed to lift this pants decomposition via $\sigma$ to the surface $M$, and we obtain a multicurve $\mu \subset M$ of total length which does not exceed $2 C_0 (2g+2) \log(2g+2)$. This multicurve is not a pants decomposition, but it is not too difficult to see that by cutting along it, one obtains a collection of cylinders, pairs of pants or four-holed spheres. This is explicitly shown in \cite[Lemma 6]{BP09}. \\

\noindent To complete the multicurve into a full pants decomposition, we must add curves which lie in the four-holed spheres. Consider the set of four-holed spheres $\{F_k\}_{k=1}^{n_0}$ which arise as connected components of $M\setminus \mu$. Note that there are at most $g-1$ of them. For each four-holed sphere $F_k$, we consider an interior curve $\gamma_k$ that cuts it into two pairs of pants. We claim that these curves can be chosen such that 
the sum of their lengths is bounded above by $C_2 \, g \log(g)$ for a universal constant $C_2$. \\

\noindent To show this, consider for each $k$, the lengths $\{\ell_{k,i}\}_{i=1}^4$ of the four boundary curves of $F_k$.
To each~$F_k$, we glue four round hemispheres of boundary lengths $\{\ell_{k,i}\}_{i=1}^4$ and we mark the four summits of the hemispheres to obtain a marked sphere $\tilde{F}_k$.
By Proposition~\ref{prop:bers}, the four-holed sphere~$\tilde{F}_k$ admits a pants decomposition (which here is reduced to a single curve) $\gamma_k$ of length at most $C_1 \sqrt{\area(\tilde{F}_k)}$.
Sliding these curves away from the marked points of the hemispheres without increasing their lengths, we can assume that each curve~$\gamma_k$ lies in~$F_k$.

If we denote $A_k=\area(F_k)$, we have that
\begin{eqnarray*}
\sum_{k=1}^{n_0} \length(\gamma_k) &\leq& C_1 \sum_{k=1}^{n_0} \sqrt{\area(\tilde{F}_k)}\\
&\leq& C_1 \sum_{k=1}^{n_0} \sqrt{A_k + \sum_{i=1}^4 \frac{\ell_{k,i}^2}{2\pi} }\\
&\leq& C_1 \sum_{k=1}^{n_0} \left(\sqrt{A_k}+ \sqrt{\sum_{i=1}^4 \frac{\ell_{k,i}^2}{2\pi} }\right)
\end{eqnarray*}
Now, unless a radicand is of value less than $1$, it bounds its square root.
Hence $\sqrt{A_k} \leq 1 + A_k$.
Denoting $L_k = \max_{i} \ell_{k,i}$, we obtain
\begin{eqnarray*}
\sum_{k=1}^{n_0} \length(\gamma_k) &\leq& C_1 \left( n_0 + \sum_{k=1}^{n_0} \left(A_k + \sqrt{ \frac{2}{\pi} L^2_k} \right) \right)\\
& \leq & C_1 \, n_0 + C_1 \sum_{k=1}^{n_0} A_k + C_1 \sum_{k=1}^{n_0} \sqrt{\frac{2}{\pi}} L_k.
\end{eqnarray*}
Note that $\sum_{k=1}^{n_0} L_k \leq 4 \sum_{k,i} \ell_{k,i} \leq 16 C_0 \, (2g+2) \log(2g+2)$ because each curve of $\mu$ can be a boundary of at most two four-holed spheres. 
Now, since $\sum_{k=1}^{n_0} A_k \leq \area(M) = 4\pi (g-1)$ and $n_0 \leq g-1$, we conclude that the sum of the lengths of the $\gamma_k$ is bounded above by $C_2 \, g \log(g)$.
Hence the claim. \\

\noindent In conclusion, the multicurve $\mu \cup \{\gamma_k\}_{k=1}^{n_0}$ contains a pants decomposition of total length not exceeding
$$2 C_0 \, (2g+2) \log(2g+2) + C_2 \, g \log(g) < C \, g \log(g)$$
for some universal constant $C$.
\end{proof}


\section{Systolic area and first Betti number of groups}

In this section, we use the approach developed in Section~\ref{sec:ell} to evaluate
the systolic area of a finitely presentable group $G$ in terms of its first Betti
number. \\

\begin{definition} \label{def:SG}
The \emph{systolic area} of~$G$ is defined as 
\begin{equation*}
\Sys(G)=\inf_X \frac{\area(X)}{\sys_\pi(X)^2},
\end{equation*}
where the infimum is taken over all piecewise flat $2$-complexes $X$ with
fundamental group isomorphic to $G$ and $\sys_\pi$ denotes the homotopical systole, \cf~Definition~\ref{def:sys}.
One can also take the infimum over piecewise Riemannian $2$-complexes since Riemannian metrics can be approximated as close as wanted by piecewise flat metrics.

Recall also that the first Betti number of~$G$ is defined as the dimension of its first real homology group
$$
H_1(G,\R):=H_1(K(G,1),\R),
$$
where $K(G,1)$ denotes the Eilenberg-MacLane space associated to $G$.
\end{definition}

\begin{theorem}
Let $G$ be a finitely presentable nontrivial group with no free factor isomorphic to~$\Z$.
Then
$$
\Sys(G)\geq C \, \frac{b_1(G)+1}{(\log(b_1(G)+2))^2}
$$
for some positive universal constant $C$.
\end{theorem}

\begin{remark}
Consider the free product $G_n=F_n * G$, where $F_n$ is the free group with $n$
generators and $G$ is a finitely presentable nontrivial group.
The first Betti number of~$G_n$ goes to infinity with~$n$, while its systolic area
remains bounded by the systolic area of~$G$.
This example shows that a restriction on the free factors is needed in the
previous theorem.
\end{remark}

\begin{proof}
Let $X$ be a piecewise flat $2$-complex with $\pi_1(X)=G$.
We can apply the metric regularization process of~\cite[Lemma~4.2]{RS08} to~$X$ and
change the metric of~$X$ for a piecewise flat metric with a better systolic area.
Thus, we can now assume from \cite[Theorem~3.5]{RS08} that the area of every
disk~$D$ of radius~$\frac{1}{8} \sys_\pi(X)$ in~$X$ satisfies
$$
\area(D) \geq \frac{1}{128} \, \sys_\pi(X)^2.
$$
We can also normalize the area of $X$ to be equal to $b_1(G)$.

If the homotopical systole of~$X$ is bounded by $1$, there is nothing to prove.
Thus, we can assume that it is greater than $\ell:=1$.
Now, set $r_0=\frac{1}{8} \sys_\pi(X)$.\\

1) Since each disk of radius~$r_{0}$ has area at least~$r_{0}^{2}/2$, the maximal
system of disjoint \mbox{$r_{0}$-disks} $\{D_{i}\}_{i \in I}$ admits at most $\frac{2 \, b_1(G)}{r_{0}^{2}}$ disks, that is,
$$
|I| \leq \frac{2 \, b_1(G)}{r_{0}^{2}}.
$$

2) As in the proof of Theorem~\ref{theo:ell} (Step~2), consider the $1$-skeleton~$\Gamma$ of
the nerve of the covering of~$M$ by the disks $2D_i+\varepsilon$ with $\varepsilon$
positive small enough.
The graph~$\Gamma$ is endowed with the metric for which each edge has length~$\ell/2$.
The map $\varphi : \Gamma \to X$, which takes each edge with endpoints $v_{i}$
and~$v_{j}$ to a segment connecting $x_{i}$ and~$x_{j}$, induces an epimorphism
$\pi_1(\Gamma) \to \pi_1(X) \simeq G$, \cf~Lemma~\ref{lem:epi} (whose proof works with complexes too). \\

3) Consider a connected subgraph $\Gamma_1$ of $\Gamma$ with a minimal number of
edges such that the restriction 
of $\varphi$ to $\Gamma_1$ still induces an epimorphism in real homology.
By Lemma~\ref{lem:isom} (whose proof works with complexes too), the homomorphism induced in real homology by the restriction of~$\varphi$ to~$\Gamma_1$ is an isomorphim (observe that $H_1(G,\R):=H_1(K(G,1),\R)\simeq H_1(X,\R)$).
Arguing as in the proof of Theorem~\ref{theo:ell}, we can show that the length
of~$\Gamma_1$ is at most $C' \, b_1(G)$ and that 
$$
\sys_\pi(\Gamma_1) \leq C'' \, \log(b_1(G)+1),
$$
where $C'=C'(\ell)$ and $C''=C''(\ell)$ are universal constants (recall that $\ell$ is fixed equal
to~$1$).
Note that a homotopical systolic loop of~$\Gamma_1$ induces a nontrivial class in real homology.
Since $\varphi$ is distance nonincreasing and its restriction to~$\Gamma_1$ induces
an isomorphism in real homology, the same upper bound holds for~$\sys_{\pi}(X)$.
Hence the result.
\end{proof}

The order of the bound in the previous theorem is asymptotically optimal, as shown by the following family of examples.

\begin{example}\label{ex:GG}
{\it Even case.}
Let $g\geq 2$ be an integer and $G_{2g}$ be the fundamental group of a closed orientable surface of genus $g$. It is a finitely presentable group with no free factor isomorphic to~$\Z$ and with first Betti number $2g$. 
The Buser-Sarnak hyperbolic surfaces~\cite{BS94} show that
$$
\Sys(G_{b}) \leq c_0 \, \frac{b}{\log(b)^2},
$$
where $b=2g$, for some positive universal constant~$c_0$. \\

\item{\it Odd case.}
Now let $G_{2g+1}$ be the fundamental group of the connected sum of a closed orientable surface of genus $g$ and a Klein bottle. It is a finitely presentable group with no free factor isomorphic to~$\Z$ and with first Betti number $2g+1$. 

Consider on the one hand a Buser-Sarnak hyperbolic surface $M$ of genus $g$ with homotopical systole greater than~$c \, \log(g)$ for some positive constant~$c$.
Consider on the other hand a flat rectangle $[0,\frac{L}{2}] \times [0,L]$ with $L = \frac{c}{2} \log(g)$, and glue the opposite sides of length~$L$ of this rectangle to obtain a flat Moebius band~$\mathcal{M}$ with boundary length~$L$.

We can find two disjoint minimizing arcs on~$M$ of length~$\frac{L}{2}=\frac{c}{4} \log(g)$.
Now, we cut the surface~$M$ open along these arcs and attach two Moebius bands~$\mathcal{M}$ along the boundary components of the surface.
We obtain a closed nonorientable surface~$M_{2g+1}$ with fundamental group isomorphic to~$G_{2g+1}$.
By construction,
$$
\area(M_{2g+1}) = 4 \pi (g-1) + \frac{c^2}{4} \log(g)^2
$$
and
$$
\sys_{\pi}(M_{2g+1}) = \frac{L}{2} = \frac{c}{4} \log(g).
$$
Thus,
$$
\Sys(G_{b}) \leq c_0 \, \frac{b}{\log(b)^2},
$$
where $b=2g+1$, for some positive universal constant $c_0$.
\end{example}

\bibliographystyle{plain}

\end{document}